\documentclass[11pt]{amsart}
\usepackage{latexsym}
\usepackage{amssymb,amsmath,mathabx, amscd}
\usepackage[pdftex]{graphicx}
\usepackage{enumerate}
\usepackage{tikz}
\usetikzlibrary{decorations.pathreplacing}
\usepackage[margin=1in]{geometry}
\usepackage{listings}
\usepackage{courier}
\usepackage{tikz-cd}
\usepackage{color}

\usepackage{MnSymbol}
\usepackage{hyperref}
\usepackage{mathrsfs, colonequals}

\newcommand{\pp}{\mathbb{P}}

\renewcommand{\H}{\mathcal{H}}

\renewcommand{\O}{\mathcal{O}}

\renewcommand{\L}{\mathcal{L}}
\newcommand{\M}{\mathcal{M}}
\newcommand{\W}{\mathcal{W}}

\newcommand{\Sym}{\operatorname{Sym}}

\newcommand{\Pic}{\operatorname{Pic}}

\newcommand{\Hilb}{\operatorname{Hilb}}

\newcommand{\End}{\operatorname{End}}

\newcommand{\ev}{\operatorname{ev}}

\newcommand{\pr}{\operatorname{pr}}

\newcommand{\eledit}[1]{{\color{teal} \sf  #1}}

\newtheorem{thm}{Theorem}[section]

\newtheorem{ithm}{Theorem}

\newtheorem{lem}[thm]{Lemma}
\newtheorem{conj}[thm]{Conjecture}

\newtheorem{cor}[thm]{Corollary}

\newcommand{\defi}[1]{\textsf{#1}} 
\newcommand{\we}{\mathcal{W}^{\vec{e}}_{\mathrm{BN}}}

\theoremstyle{definition}
\newtheorem{defin}[thm]{Definition}

\theoremstyle{remark}
\newtheorem{rem}[thm]{Remark}

\title{The embedding theorem in Hurwitz--Brill--Noether Theory}
\author{Kaelin Cook-Powell, David Jensen, Eric Larson, Hannah Larson, and Isabel Vogt}

\begin{document}

\maketitle

\begin{abstract}
We generalize the Embedding Theorem of Eisenbud--Harris from classical Brill--Noether theory
to the setting of Hurwitz--Brill--Noether theory.

More precisely, in classical Brill--Noether theory, the embedding theorem states that
a general linear series of degree \(d\)
and rank \(r\) on a general curve of genus \(g\) is an embedding if \(r \geq 3\).
If \(f \colon C \to \pp^1\) is a general cover of degree \(k\),
and \(\mathcal{L}\) is a line bundle on \(C\),
recent work of the authors shows that the splitting type of \(f_* \mathcal{L}\)
provides the appropriate generalization of the pair \((r, d)\) in classical Brill--Noether theory 
\cite{CPJ1, CPJ2, LLV, HannahRefined}.

In the context of Hurwitz--Brill--Noether theory, the condition \(r \geq 3\) is no longer sufficient
to guarantee that a general such linear series is an embedding.
We show that the additional condition needed to guarantee that a general linear series \(|\L|\) is an embedding is that the splitting type of \(f_* \mathcal{L}\) has at least three nonnegative parts.  This new extra condition reflects the unique geometry of \(k\)-gonal curves, which lie on scrolls in \(\pp^r\).
\end{abstract}

\section{Introduction}

Brill--Noether theory provides the bridge between the classical perspective on curves as subsets of projective space and the modern theory of abstract curves.
Given an algebraic curve $C$, the line bundles that could give rise to a degree \(d\) explicit realization of \(C\) in \(\pp^r\) are parameterized by its \defi{Brill-Noether locus}
\[
W^r_d (C) \colonequals \{ \L \in \Pic^d (C) \mbox{ } \vert \mbox{ } h^0 (\L) \geq r+1 \} .
\]
For \(C\) a general curve, the geometry of \(W^r_d(C)\) is well-understood by the main theorems of classical Brill--Noether theory, established in a series of papers from the 1980s \cite{ im, fl, gp, bn,  kempf, kl}.  In particular, \(W^r_d(C)\) is nonempty if and only if \(\rho(g,r,d) \colonequals g-(r+1)(g-d+r)\) satisfies \(\rho \geq 0\).
In this case, the universal $\mathcal{W}^r_d$ has a unique irreducible component $\mathcal{W}^r_{d, \text{BN}}$ dominating the moduli space $\M_g$.

Equipped with a good understanding of the moduli space $\mathcal{W}^r_{d, \text{BN}}$, one can return to the original motivation and ask finer questions about the map to projective space corresponding to a general line bundle in $\mathcal{W}^r_{d, \text{BN}}$.  The first natural such question is whether a general such line bundle
defines an \emph{embedding} into projective space, that is, when is it very ample?
In their \emph{Embedding Theorem} \cite[Theorem 1]{EH83}, Eisenbud--Harris show that a general line bundle in $\mathcal{W}^r_{d, \text{BN}}$ is very ample if $r \geq 3$.

\begin{rem}
In addition to being sufficient, the condition $r \geq 3$ is almost necessary. More precisely, there are only four triples $(g, r, d)$ with $r < 3$ where the general $\L \in \mathcal{W}^r_{d, \text{BN}}$ is very ample, namely \((g, r, d) \in \{ (0, 1,1), (0, 2, 2), (1, 2, 3), (3, 2, 4) \} \).
\end{rem}

In this article, we consider the analogous problem when the curve $C$ is equipped with a fixed degree \(k\) map $f \colon C \to \pp^1$.  In this setting, the analogues of Brill--Noether loci are the Brill--Noether splitting loci, whose definition we now recall.  Given a vector $\vec{e} = (e_1 , \ldots , e_k)$ with \(e_1 \leq e_2 \leq \cdots \leq e_k\), we define the vector bundle $\O (\vec{e}) \colonequals \bigoplus_{i=1}^k \O_{\pp^1} (e_i)$.  We then define the \defi{Brill--Noether splitting locus}
\[
W^{\vec{e}} (C,f) \colonequals \left\{ \L \in \Pic (C) \mbox{ } \vert \mbox{ } f_* \L \simeq \O (\vec{e}) \text{ or a specialization thereof}\right\} .
\]
 The expected dimension of $W^{\vec{e}} (C,f)$ is given by
\[
\rho' (g, \vec{e}) \colonequals g-  \sum_{i,j} \max \{ 0 , e_j - e_i - 1 \} .
\]
When \((C, f)\) is general, \(W^{\vec{e}}(C, f)\) is nonempty if and only if $\rho'(g, \vec{e}) \geq 0$ \cite{CPJ1, CPJ2, HannahRefined}.
In this case, when the characteristic of the ground field is \(0\) or greater than \(k\), the universal $\mathcal{W}^{\vec{e}}$ has a unique irreducible component $\we$ dominating the Hurwitz space $\H_{g,k}$ \cite{LLV}.  Our main results determine 
when a general line bundle in this component is basepoint free or very ample. (See Remark \ref{char} for the situation when the characteristic is positive but less than or equal to $k$.)

\begin{ithm}\label{Thm:bpf}
A general line bundle $(f, \L) \in \we$ is basepoint free if and only if either \(e_{k-1} \geq 0\), or $\L \simeq f^* \O_{\pp^1}(n)$ for $n \geq 0$.
\end{ithm}

\begin{ithm}
\label{Thm:VA-intro}
A general line bundle in $\we$ is  very ample if $e_{k-2} \geq 0$ and $r = h^0(\O(\vec{e})) - 1 \geq 3$.
\end{ithm}

\begin{rem}
As in the original Embedding Theorem, the sufficient condition in Theorem \ref{Thm:VA-intro} is almost necessary. We make this precise by giving the necessary and sufficient conditions for very ampleness in Theorem \ref{Thm:VA}.
\end{rem}

Taking $k$ sufficiently large, the condition $r \geq 3$ in Theorem \ref{Thm:VA-intro} implies the original Embedding Theorem of Eisenbud--Harris.
In the regime of small $k$, another condition is needed to capture the unique geometry of \(k\)-gonal curves:
The condition involving the number of nonnegative parts of \(\vec{e}\) reflects the fact that such maps necessarily factor through projective bundles
over \(\pp^1\).

Correspondingly, our approach to this problem splits into two parts:
One involving the ``ampleness of the map from the curve to the projective bundle'',
and a second involving ``ampleness of the map from the projective bundle to the projective space along the curve''.

More generally, we can completely understand the ``degree of ampleness'' of the first of these maps.
Recall that a line bundle $\L$ on a curve $C$ is called \defi{$p$-very ample} if, for every effective divisor $D$ on $C$ of degree $p+1$, we have $h^0 (\L(-D)) = h^0 (\L) - (p+1)$.  Note that a line bundle is 0-very ample if and only if it is basepoint free, and it is 1-very ample if and only if it is very ample.  In this way, the notion of $p$-very ampleness is a natural generalization
of both basepoint freeness and very ampleness.

In the setting of curves with a fixed map to $\pp^1$, it is simpler to study a variant of $p$-very ampleness that is relative to the map to $\pp^1$.  Given a cover $f \colon C \to \pp^1$, we say that a divisor $D$ on $C$ is \defi{fibral} if \(D\) is effective and supported in a fiber of \(f\). We make the following definition.

\begin{defin} \label{def:rel-va}
Let $f\colon C \to \pp^1$ be a cover.  We say that a line bundle $\L$ on $C$ is \defi{relatively $p$-very ample} if, for every fibral divisor $D$ of degree $p+1$ on $C$, we have
\[
h^0 (\L(-D)) = h^0 (\L) - (p+1) .
\]
Similarly, we say that $\L$ is \defi{birationally relatively $p$-very ample} if the above holds for all but finitely many fibral divisors $D$ of degree $p+1$ on $C$.
\end{defin}
\noindent
This definition is only useful when $p \leq k - 1$, since for $p \geq k$, there are no degree $p+1$ fibral divisors.

\begin{ithm} \label{bi-rel-va-intro} Assume $p \leq k - 1$.
A general line bundle in \(\we\) is birationally relatively \(p\)-very ample if and only if $e_{k-p} \geq 0$. 
\end{ithm}
\begin{ithm} \label{rel-va-intro} Assume $p \leq k - 1$.
A general line bundle in  \(\we\) is relatively \(p\)-very ample if $e_{k-p-1} \geq 0$.
\end{ithm}

\begin{rem}
Again, the condition in Theorem \ref{rel-va-intro} is almost necessary. We make this precise in Theorem \ref{rel-va}, where we give the necessary and sufficient conditions for relative $p$-very ampleness.
\end{rem}

Our proofs proceed by embedded degeneration.
We begin in Section~\ref{Sec:Prelim} by recalling the correspondence between line bundles 
\(\mathcal{L} \in W^{\vec{e}}(C, f)\) and maps \(C \to \pp\O(\vec{e})^\vee\);
this extends the correspondence in classical Brill--Noether theory between
line bundles \(\L \in W^r_d(C)\) and maps \(C \to \pp H^0(\L)^\vee\). 
Our theorems can therefore be approached by constructing suitable curves \(C \subset \pp\O(\vec{e})^\vee\).
This we do inductively:
In Section~\ref{sec:our_degeneration},
we explain how to start with a suitable curve 
\[C_{k - 1} \subset \pp\O(-e_1, -e_2, \ldots, -e_{k-1}) \subset \pp\O(\vec{e})^\vee,\]
and attach lines and a section to form a certain reducible curve \(X \subset \pp\O(\vec{e})^\vee\).
Even though the curve \(X\) does not satisfy the necessary properties,
we are able to analyze the deformations of \(X\) in Sections~\ref{Sec:Smooth}--\ref{Sec:VA}
to show that a suitably general deformation does.

Simply by considering the curves we construct in this manner,
we obtain a 6-page ``constructive'' proof of
Hurwitz--Brill--Noether existence (see Remark~\ref{existence}).
Even just this result is highly nontrivial: The original proof by Jensen--Ranganathan
of a special case \cite{JR} was quite involved,
and subsequent complete proofs have relied on deep results about either tropical geometry \cite{CPJ1},
intersection theory on the moduli stack of vector bundles \cite{P1},
or analogs of the regeneration theorem \cite{LLV}.
Moreover, the constructive nature of this proof opens up the door to probing
many other aspects of the geometry of $k$-gonal curves; its
usefulness is unlikely to be limited to only the study of ampleness given here.

\begin{rem}[Characteristic hypotheses] \label{char}
When the characteristic of the ground field is zero or greater than \(k\), the Hurwitz space $\H_{k,g}$ is irreducible.
By contrast, when the characteristic of the ground field is positive but less than or equal to $k$, the 
Hurwitz space $\H_{k,g}$ is not known to be irreducible.
Nevertheless, our proofs show that there exists a component of $\W^{\vec{e}}$ dominating \emph{some} component of the Hurwitz space, in which the statements of the theorems hold.
When we say that $f\colon C \to \pp^1$ is a general cover, we shall mean that it is general in the appropriate component of the Hurwitz space.

When the characteristic of the ground field is positive but less than or equal to $k$, the main theorems of \cite{LLV} also apply to some component of the Hurwitz space (see \cite[Remark 1]{LLV}). 
A priori, this component need not be the same as any of our components in the present paper. Moreover, the components here need not be the same for different $\vec{e}$.
\end{rem}

\subsection*{Acknowledgements}
During the preparation of this article,
D.J.\ was supported by NSF grant DMS-2054135,
E.L.\ was supported by
NSF grants DMS-1802908 and DMS-2200641, H.L.\ was supported by the Clay Research Fellowship, and I.V.\ was supported by NSF grants DMS-1902743 and DMS-2200655.

\section{Preliminaries on projective bundles and splitting types}
\label{Sec:Prelim}

 Throughout, we  denote our splitting type by $\vec{e} = (e_1, \ldots, e_k)$ with $e_1 \leq \cdots \leq e_k$.  
We write \(\deg(\vec{e}) \colonequals e_1 + \cdots + e_k\).  
 We define
 \[u(\vec{e})\colonequals h^1(\pp^1, \End(\O(\vec{e}))) = \sum_{i < j} \max\{0, e_j - e_i - 1\}.\]
 Note that $\rho'(g, \vec{e}) = g - u(\vec{e}).$
 Let $f \colon C \to \pp^1$ be a degree $k$, genus $g$ cover, and suppose 
 \[
 f_*\L = E \simeq \O(\vec{e}) .
 \]
Pulling back to $C$, there is a natural surjection
 \[f^*E = f^*f_*\L \to \L,\]
 which corresponds to evaluation of sections on a fiber at a point of $C$.  Dualizing, we obtain an injection
 \[\L^\vee \hookrightarrow f^*E^\vee \]
with locally free cokernel. This defines an embedding 
\begin{equation} \label{CintoPE}
\begin{tikzcd}
C \arrow[hookrightarrow]{r}{\iota} \arrow{rd}[swap]{f} &\pp E^\vee \arrow{d}{\pi} \\
& \pp^1
\end{tikzcd}
\end{equation}
such that $\iota^*\O_{\pp E^\vee}(1) = \L$. Throughout, we use the subspace convention for projective bundles, so that $\pi_* \O_{\pp E^\vee}(1) = E$.

\begin{lem} \label{good-embeddings}
For every $t \in \pp^1$, the fiber $f^{-1}(t)$ is a basis for $\pi^{-1}(t) \simeq \pp^{k-1}$. Conversely, suppose we have an embedding $\iota\colon C \to \pp E^\vee$ factoring $f$, as in \eqref{CintoPE}. If $f^{-1}(t)$ is a basis for $\pi^{-1}(t)$ for all $t \in \pp^1$, then
\(f_* \iota^*\O_{\pp E^\vee}(1) \simeq E\).
\end{lem}
\begin{proof}
To say $f^{-1}(t)$ is a basis for $\pi^{-1}(t) \simeq \pp^{k-1}$ is to say that the evaluation map 
\[ H^0(\O_{\pp E^\vee}(1)|_{\pi^{-1}(t)}) \to \O_{\pp E^\vee}(1)|_{f^{-1}(t)} \]
is an isomorphism.  This in turn is equivalent to saying that the map
\[
E = \pi_* \O_{\pp E^\vee}(1) \to f_* \iota^*\O_{\pp E^\vee}(1) 
\]
is an isomorphism at $t$.
\end{proof}

The bundle \(E^\vee\) admits a filtration
\begin{equation} \label{filt}
0 \subset \O(-e_1) \subset \O(-e_1, -e_2) \subset \cdots \subset \O(-e_1, -e_2, \ldots, -e_{k-1}) \subset \O(-e_1, -e_2, \ldots , -e_{k-1}, -e_k) = E^\vee.
\end{equation}
If the splitting type has distinct parts, this filtration is unique (it is the Harder--Narasimhan filtration); otherwise, we choose such a filtration.
We define $F_i \colonequals  \O(-e_1, -e_2, \ldots, -e_i)$ to be the $i$th bundle above, so \eqref{filt} becomes
\[
0 \subset F_1 \subset \cdots \subset F_{k-1} \subset F_k = E^\vee .
\]
For ease of notation, write \(\Sigma_i \colonequals \pp F_i\).
The above filtration of \(E^\vee\) gives a sequence of inclusions
\[\Sigma_1 \subset \Sigma_2 \subset \cdots \subset \Sigma_{k - 1} \subset \Sigma_k = \pp E^\vee.\]
We also define
\[F_- \colonequals F_{\#\{i : e_i < 0\}} \simeq \bigoplus_{i : e_i < 0} \O(-e_i)
\quad \text{and} \quad
F_0 \colonequals F_{\#\{i : e_i \leq 0\}} \simeq  \bigoplus_{i : e_i \leq 0} \O(-e_i).
\]

Projection away from $\pp F_- \subset \pp E^\vee$ defines a rational map which we denote $\pr_-\colon \pp E^\vee \dashedrightarrow \pp(E^\vee/F_-)$.
The line bundle $\O_{\pp E^\vee}(1)$ on $\pp E^\vee$ defines a rational map $\pp E^\vee \dashedrightarrow \pp^r$, where
\[
r + 1 =  h^0(\pp^1, E) = h^0(\pp E^\vee, \O_{\pp E^\vee}(1))  = \sum_{i=1}^k \max\{0, e_i + 1\}.
\]
This map factors through $\pr_-$.  In particular,
the complete linear series $|\L| \colon C \to \pp^r$ factors as
\begin{center}
\begin{tikzcd}
C \arrow[hook]{r}{\iota} & \pp E^\vee \arrow[dashed]{r}{\pr_-} & \pp(E^\vee/F_-) \arrow{r} &\pp^r.
\end{tikzcd}
\end{center}
Thus, we break our problem into two parts: First we study the map $(\pr_- \circ \iota) \colon C \to \pp(E^\vee/F_-)$.  Then we study $\pp(E^\vee/F_-) \to \pp^r$ along the image of the first map.
The image of $ \pp(E^\vee/F_-) \rightarrow \pp^r$ is a cone over $\pp(E^\vee/F_0) \hookrightarrow \pp^r$\eledit{,} with vertex a linear space which is the image of  $\pp(F_0/F_-)$.

The definition of (birational) relative very ampleness given in Definition~\ref{def:rel-va}
is the analog of (birational) very ampleness for the first map
\(\pr_- \circ \iota\).
Our first task is, thus, to determine exactly when
\(\L\) is (birationally) relatively very ample.

\section{Overview of inductive strategy}
\label{sec:our_degeneration}

The following theorem is the basis for our inductive strategy to prove the main theorems. 

\begin{thm}\label{thm:main_inductive}
Let \(k \geq 1\) and let \(E = \O(\vec{e})\).  
There exists a smooth curve \(C \subset \pp E^\vee\) of any genus \(g\) such that $\rho'(g, \vec{e}) \geq 0$ (with \(g=u(\vec{e}) = 0\) when \(k=1\)) satisfying all of the following conditions:
\begin{enumerate}
\item\label{main_inductive:pushforward} \((\pi|_C)_*\O(1)|_C \simeq E\),
\item\label{main_inductive:h1vanishes} \(h^1(N_{C/ \pp E^\vee}) = 0\),
\item \label{main_inductive:general} the map \(\pi|_C \colon C \to \pp^1\) is general in the Hurwitz space, and
\item \label{main_inductive:subscroll} the curve \(C\) meets \(\Sigma_{k-1}\) at finitely many points, and does not meet \(\Sigma_{k-2}\).
\end{enumerate}
\end{thm}

\begin{rem} \label{existence}
Combining Lemma \ref{good-embeddings} with Theorem \ref{thm:main_inductive} parts \eqref{main_inductive:pushforward} and \eqref{main_inductive:general} gives a short alternative proof of the existence theorem in Hurwitz--Brill--Noether theory, i.e., that $W^{\vec{e}}(C, f)$ is non-empty when $\rho'(g, \vec{e}) \geq 0$.
\end{rem}

Our proof of Theorem \ref{thm:main_inductive} will be by induction on the rank \(k\) of \(E\).  For the base case, $k=1$, the curve $C$ is the unique section of the $\pp^0$-bundle $\pp F_1$; in this case, the genus of $C$ is zero. For $k \geq 2$, we will construct a curve satisfying Theorem~\ref{thm:main_inductive} of arbitrary genus \(g \geq u(\vec{e})\).

Write \(\vec{e}_{<k} = (e_1, e_2, \ldots, e_{k - 1})\) for the splitting type of \(F_{k-1}^\vee\), and choose a direct sum decomposition \(E^\vee \simeq F_{k-1} \oplus \O(-e_k)\).    (As mentioned in Section~\ref{Sec:Prelim}, if \(e_k > e_{k-1}\), then the bundle \(F_{k-1}\) is canonical; otherwise we make a choice.)  We have an inclusion \(\Sigma_{k-1} = \pp F_{k-1} \hookrightarrow \Sigma_k = \pp E^\vee\) as a divisor.  By induction on \(k\), there exists  a smooth curve \(C_{k-1}\) in \(\Sigma_{k-1}\) of genus \(u(\vec{e}_{<k})\) satisfying Theorem~\ref{thm:main_inductive}.

For the inductive step, for any genus \(g \geq u(\vec{e})\), we construct a nodal curve \(X\) of arithmetic genus \(g\) as the union of such a curve \(C_{k-1} \subset \Sigma_{k-1}\), a section \(S_k\) corresponding to a choice
of splitting \(\O(-e_k) \hookrightarrow E^\vee\),
and a collection of lines connecting \(S_k\) to \(C_{k - 1}\).  More precisely,
define
\begin{equation} \label{mdef}
m \colonequals g+1 - u(\vec{e}_{<k}).
\end{equation}
 Choose \(m\) general points \(p_1, \dots, p_m\) on \(C_{k-1}\) over general points \(t_1, \dots, t_m\) in \(\pp^1\).  Write \(q_i\) for the unique point on \(S_k\) over \(t_i\), and let \( \Gamma \colonequals p_1+ \cdots + p_m\).  We will write \(L_i\) for the line in the fiber \(\pp E^\vee|_{t_i}\) joining \(p_i\) and \(q_i\).  
Define
\[X \colonequals C_{k-1} \cup S_k \cup L_1 \cup \cdots \cup L_m.\]
By construction, \(X\) has arithmetic genus \(g\). The following diagram illustrates \(X\):

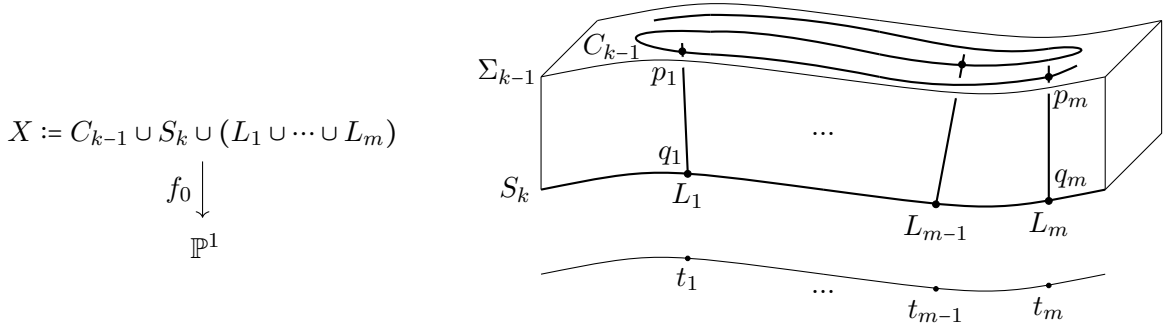
\begin{center}
\begin{tikzpicture}[scale=1.5]
\draw (-3,2.5) node {\(X \colonequals C_{k-1} \cup S_k \cup (L_1 \cup \cdots \cup L_m)\)};
\draw[thick] (1.3,2.15) -- (1.25,3.3);
\draw[thick] (3.5,1.865) -- (3.75,3.2);
\draw[thick] (4.5,1.895) -- (4.5,3.1);
\draw[thick, below] (1.3,2.15) node {\(L_1\)};
\draw[thick, below] (3.5,1.865) node {\(L_{m-1}\)};
\draw[thick, below] (4.5,1.895) node {\(L_m\)};
\filldraw[white] (3.7,2.9) circle (2.5pt);
\filldraw[white] (1.25, 3.13) circle (1.25pt);
\filldraw[white] (4.5, 2.9) circle (1.25pt);
\draw[ left] (0,2) node {\( S_k\)};
\draw[ thick] (0, 2) .. controls (1,2.2) .. (2.5, 2);
\draw[ thick] (2.5, 2) .. controls (4,1.8) .. (5, 2);
\draw (5, 2) -- (5.5, 2.5) -- (5.5, 3.5);
\draw (0, 2) -- (0, 3);
\draw (5, 2) -- (5, 3);
\draw  (0, 2+1) .. controls (1,2.2+1) .. (2.5, 2+1);
\draw  (2.5, 2+1) .. controls (4,1.8+1) .. (5, 2+1);
\draw (0+.5, 2+1+.5) .. controls (1+.5,2.2+1+.5) .. (2.5+.5, 2+1+.5);
\draw (2.5+.5, 2+1+.5) .. controls (4+.5,1.8+1+.5) .. (5+.5, 2+1+.5);
\draw  (5+.5, 2+1+.5) -- (5, 2+1);
\draw  (0, 2+1) -- (0+.5, 2+1+.5);
\draw (0, 2-.75) .. controls (1,2.2-.75) .. (2.5, 2-.75);
\draw (2.5, 2-.75) .. controls (4,1.8-.75) .. (5, 2-.75);
\draw[left] (-2.75,1.5) node {\(\pp^1\)};
\draw[->] (-3,2.25) -- (-3,1.75); 
\draw (-3,2) node[left]{$f_0$};
\draw[ thick] (1, 3.5) .. controls (1.1, 3.5+.1*.2) and (1.4, 3.55) ..(1.5, 3.55); 
\draw[ thick] (1.5, 3.55) .. controls (2, 3.55) and (2.5, 3.4 + .5*.2) .. (3, 3.4); 
\draw[ thick] (3, 3.4) .. controls (3.5, 3.4 - .5*.2) and (4, 3.25 - .5*.15) .. (4.5, 3.25);
\draw[ thick] (4.5, 3.25) .. controls (4.75, 3.25 + .25*.15) and (5, 3.15 + .5*.18) .. (4.5, 3.15);
\draw[ thick] (4.5, 3.15) ..  controls (4, 3.15 - .5*.18) and (3.5, 3.2 - .5*.2).. (3, 3.2);
\draw[ thick] (3, 3.2) .. controls (2.75, 3.2 + .25*.2) and (2, 3.4 + .5*.05).. (1.5, 3.4);
\draw[thick] (1.5, 3.4) .. controls (.5, 3.5 - .05) and (.75, 3.2+.75*.05) .. (1.5,3.2);
\draw[ thick] (1.5,3.2) .. controls (2, 3.2-.5*.05) and (2.5, 3 +.5*.2) .. (3,3);
\draw[ thick] (3,3) .. controls (3.5, 3 -.5*.2) and (4, 3 - .5*.18) .. (4.5,3);
\draw[ thick] (4.5,3) .. controls (4.55, 3 + .05*.18) and (4.7, 3.1 - .05*.5) .. (4.75, 3.1);
\draw[thick] (2.5, 2.5) node {...};
\draw[thick] (.63, 3.27) node{\(C_{k-1}\)};
\draw[thick] (-0.3, 3.05) node{\(\Sigma_{k-1}\)};
\filldraw (1.252,3.227) circle[radius=0.03];
\filldraw (3.731,3.11) circle[radius=0.03];
\filldraw (4.5,3) circle[radius=0.03];
\draw (1.1,3) node{\(p_1\)};
\draw (4.7,2.8) node{\(p_m\)};
\filldraw (1.3,2.14) circle[radius=0.03];
\filldraw (3.5,1.87) circle[radius=0.03];
\filldraw (4.5,1.9) circle[radius=0.03];
\draw (1.15,2.3) node{\(q_1\)};
\draw (4.7,2.1) node{\(q_m\)};
\filldraw (1.3,2.14-.75) circle[radius=0.02];
\filldraw (3.5,1.87-.75) circle[radius=0.02];
\filldraw (4.5,1.9-.75) circle[radius=0.02];
\draw[thick, below] (1.3,2.15-.75) node {\(t_1\)};
\draw[thick, below] (3.5,1.865-.75) node {\(t_{m-1}\)};
\draw[thick, below] (4.5,1.895-.75) node {\(t_m\)};
\draw[thick] (2.5, 1.1) node {...};
\end{tikzpicture}
\end{center}
In subsequent sections, we prove that a general deformation of $X$ satisfies Theorem \ref{thm:main_inductive}.

\section{The normal bundle of \(X\): smoothing the nodal curve}
\label{Sec:Smooth}

In this section, we show that there exists a deformation of $X$ smoothing all the nodes.  As a consequence, we obtain
Theorem~\ref{thm:main_inductive} parts~\eqref{main_inductive:pushforward} and~\eqref{main_inductive:h1vanishes}.  Recall that the deformation theory of $X$ in $\Sigma_k$ is controlled by its normal bundle $N_X \colonequals N_{X/\Sigma_k}$.  For this reason, we now prove several results on the normal bundle of \(X\) restricted to various components.  

\subsection{\boldmath The normal bundles of \(\Sigma_{k-1}\) and \(S_k\)}
We first record some elementary results about the normal bundles of the two projective subbundles \(\Sigma_{k-1}\) and \(S_k\) in \(\Sigma_k\).
Since these subbundles correspond to the direct sum decomposition
\(E^\vee \simeq F_{k - 1} \oplus \O_{\pp^1}(-e_k)\), we have
\begin{align}
N_{\Sigma_{k-1}} &\simeq \O_{\pp E^\vee}(1)|_{\Sigma_{k-1}} \otimes \pi^* \O_{\pp^1}(-e_k) \label{eq:normal_bundle_sigma} \\
N_{S_k} &\simeq \O_{\pp E^\vee}(1)|_{S_k} \otimes \pi^* F_{k-1} \simeq \bigoplus _{i < k} \O_{\pp^1}(e_k - e_i). \label{nsk}
\end{align}
In particular, since \(e_k \geq e_i\) for all \(i\leq k\), we obtain
\begin{equation}\label{eq:h1NS}
H^1(N_{S_k}) = 0.
\end{equation}

\subsection{\boldmath The normal bundle of \(C_{k-1}\)}\label{sec:N_C}
The normal bundle exact sequence
\begin{equation}
\label{eq:standard}
0 \to N_{C_{k-1}/\Sigma_{k-1}} \to N_{C_{k-1}} \to N_{\Sigma_{k-1}}|_{C_{k-1}} \to 0,
\end{equation}
for \(C_{k-1} \subset \Sigma_{k-1} \subset \pp E^\vee\), allows us to combine results about \(N_{\Sigma_{k-1}}\) with our inductive hypotheses about \(C_{k-1} \subset \Sigma_{k-1}\) to deduce information about \(N_{C_{k-1}}\).
By induction we have \(H^1(N_{C_{k-1}/\Sigma_{k-1}}) = 0\) and
\((\pi|_{C_{k-1}})_* \O_{\pp E^\vee}(1)|_{C_{k-1}} \simeq F_{k-1}^\vee\).  Hence combining \eqref{eq:normal_bundle_sigma} with the push-pull formula, we have
\begin{equation*}
(\pi|_{C_{k-1}})_*N_{\Sigma_{k-1}}|_{C_{k-1}} \simeq F_{k-1}^\vee(-e_k) \simeq \bigoplus_{i < k} \O_{\pp^1}(e_i - e_k) .
\end{equation*}
In particular, we have 
\begin{equation}\label{eq:h1_NF}
h^1(N_{\Sigma_{k-1}}|_{C_{k-1}}) = \sum_{i < k} \max \{0, e_k - e_i - 1\}.
\end{equation}

The restricted normal bundle \(N_X|_{C_{k-1}}\) is a positive elementary modification of \(N_{C_{k-1}}\) at the points \(\Gamma = p_1+ \dots + p_m\).
See \cite[Sections 3.1--3.2]{interpolation} for a brief introduction to elementary modifications of normal bundles.
These modifications are described in the following lemma.

\begin{lem}\label{lem:Nx_restC}
The restricted normal bundle \(N_X|_{C_{k-1}}\) sits in the exact sequence
\begin{equation}\label{eq:NX_restricted_C}
0 \to N_{C_{k-1}/\Sigma_{k-1}} \to N_{X}|_{C_{k-1}} \to N_{\Sigma_{k-1}}|_{C_{k-1}}(\Gamma) \to 0.
\end{equation}
\end{lem}
\begin{proof}
This follows from the normal bundle exact sequence \eqref{eq:standard} and the fact that each line \(L_i\) meets \(\Sigma_{k-1}\) transversely at \(p_i\) as in \cite[Equation (7)]{interpolation}.
\end{proof}

\subsection{\boldmath The normal bundles of the \(L_i\)}\label{sec:N_lines}
The normal bundle of each component \(L_i \subset X\) is also simple to describe.  Write \(\pp E_{t_i}^\vee\) for the fiber of \(\pp E^\vee\) over \(t_i \in \pp^1\).  Then the restricted normal bundle sits in an exact sequence
\begin{equation}\label{eq:normal_line_unmod}
0 \to [N_{L_i/\pp E_{t_i}^\vee} \simeq \O(1)^{\oplus (k - 2)}] \to N_{L_i} \to [N_{\pp E_{t_i}^\vee}|_{L_i} \simeq \O] \to 0.
\end{equation}
Again, the restriction of \(N_X\) to one of the line components \(L_i\) is a positive modification of \(N_{L_i}\).
Since \(C_{k-1} \cup S_k\) meets \(\pp E_{t_i}^\vee\) transversely at \(p_i\) and \(q_i\), we obtain the exact sequence
\begin{equation} \label{lem:NX_restL}
0 \to [N_{L_i/\pp E_{t_i}^\vee} \simeq \O(1)^{\oplus k-2}] \to N_{X}|_{L_i} \to [N_{\pp E_{t_i}^\vee}|_{L_i}(p_i + q_i) \simeq \O(2)] \to 0.
\end{equation}

\subsection{\boldmath Deformations of \(C_{k-1}\) transverse to \(\Sigma_{k-1}\)}
In this section we prove the following lemma.

\begin{lem}
\label{Cor:Surj}
The following map is surjective:
\[
H^0 (N_X) \to H^0 (N_{\Sigma_{k-1}} \vert_{C_{k-1}} (\Gamma)).
\]
\end{lem}
\begin{proof}
The map factors as
\[H^0(N_X) \to H^0(N_X|_{C_{k-1}}) \to H^0(N_{\Sigma_{k-1}}|_{C_{k-1}}(\Gamma)).\]
We consider each of the above maps in turn. The second map appears in the long exact sequence associated to the exact sequence~\eqref{eq:NX_restricted_C}.  By induction, we have $H^1 (N_{C_{k-1}/\Sigma_{k-1}}) = 0$, and surjectivity follows.

Next we show that $H^0(N_X) \to H^0(N_X|_{C_{k-1}})$ is surjective.
By considering the exact sequence for restriction to \(C_{k-1}\):
\[
0 \to N_X \vert_{S_k \cup L_1 \cup \cdots \cup L_m} (- \Gamma) \to N_X \to N_X \vert_{C_{k-1}} \to 0 ,
\]
it suffices to show that $H^1 (N_X \vert_{S_k \cup L_1 \cup \cdots \cup L_m} (- \Gamma)) = 0$.  To see this, restrict further to \(S_k\):
\begin{equation} \label{lns}
0 \to \bigoplus_{i=1}^m N_X \vert_{L_i} (-p_i - q_i ) \to N_X \vert_{S_k \cup L_1 \cup \cdots \cup L_m}(-\Gamma) \to N_X \vert_{S_k} \to 0 .
\end{equation}
By \eqref{lem:NX_restL},
we have $h^1 (N_X \vert_{L_i} (-p_i -q_i )) = 0$.  
Finally, \(N_X|_{S_k}\) is a positive modification of \(N_{S_k}\).  By \eqref{eq:h1NS}, we have \(H^1(N_{S_k}) = 0\); hence \(H^1(N_X|_{S_k}) = 0\).
\end{proof}

\subsection{\boldmath Smoothing the nodes of \(X\)}
Recall that if \(X\) is a nodal curve and \(p\) is a node of \(X\), we have the exact sequence of sheaves
\[0 \to N_X^p \hookrightarrow N_X \xrightarrow{\ev_p} T_p^1 \to 0,\]
which plays a crucial role in the deformation theory of \(X\).  The sheaf \(T_p^1\) is a skyscraper sheaf supported at \(p\). The subsheaf \(N_X^p\) of $N_X$ corresponds to deformations of \(X\) failing to smooth the node \(p\).  Since \(T_p^1\) is supported at \(p\), the sheaves \(N_X\) and \(N_X^p\) are isomorphic away from \(p\).  Likewise, for any finite subset \(\Delta\) of the nodes of \(X\), we can define \(N_X^\Delta\).  (When \(\Delta\) is the entire singular locus of \(X\), this sheaf is known as the \defi{equisingular normal sheaf}, cf.\ \cite[Section 4.7.1]{sernesi}.)

\begin{lem}\label{lem:smooth_defs_exist}
For any points \(p = p_i\) and \(q = q_i\), we have
\[H^1(N_X^{p,q}) = 0.\]
In particular, there exists a deformation of \(X\) smoothing all of the nodes.
\end{lem}
\begin{proof}
Restriction to the (disconnected) curve \(C_{k-1} \cup S_k\) gives an exact sequence
\[ 0 \to \bigoplus_{j=1}^m N_X|_{L_j}(-p_j - q_j) \to N_X^{p,q} \to N_X^p|_{C_{k-1}} \oplus N_X^q|_{S_k} \to 0.\]
Making modifications at each of the points \(\{p_1, \dots, p_m\}\) except \(p\), the exact sequence
of Lemma \ref{lem:Nx_restC} induces the exact sequence
\[0 \to N_{C_{k-1}/\Sigma_{k-1}} \to N_X^p|_{C_{k-1}} \to N_{\Sigma_{k-1}}|_{C_{k-1}}(p_1 + \cdots + \hat{p} + \cdots + p_m) \to 0.\]
The subbundle has no higher cohomology by induction (Theorem \ref{thm:main_inductive}\eqref{main_inductive:h1vanishes}).  By \eqref{eq:h1_NF}, 
\[ h^1(N_{\Sigma_{k-1}}|_{C_{k-1}})  = \sum_{i \neq k} \max \{ 0, e_k  - e_i - 1 \} \leq m - 1.\]
Since twisting up by a general point \(p_i\) decreases the \(h^1\) by \(1\), the quotient has no higher cohomology as well.
Thus \(H^1(N_X^{p,q}) = 0\) as desired.

The vanishing of \(H^1(N_X^{p,q})\) implies that each of the evaluation maps 
\[H^0(N_X) \xrightarrow{\ev_{p,q}} T^1_{p,q}\]
is surjective.  A general section of \(N_X\) therefore smooths all of the nodes of \(X\).
\end{proof}

\subsection{Proof of Theorem~\ref{thm:main_inductive} parts~\eqref{main_inductive:pushforward} and~\eqref{main_inductive:h1vanishes}}
By Lemma \ref{lem:smooth_defs_exist}, smooth deformations of the curve \(X\) exist.  We will show that a general such deformation \(C\) has all of the desired properties in Theorem \ref{thm:main_inductive}. 

\begin{lem}
\label{Lem:2PointsOnALine}
Under specialization to $f_0\colon X \to \pp^1$, the limit of a fiber of \(f \colon C \to \pp^1\) is either
\begin{itemize}
\item a fiber of $f_0 \colon X \to \pp^1$ over a point $t \notin \{t_1, \ldots, t_m\}$, or
\item contained in the fiber of $f_0$ over $t_i$ and contains exactly two points on $L_i$.
\end{itemize}
In particular, the limit of a
fibral divisor on $C$ contains at most two points on any $L_i$.
\end{lem}

\begin{proof}
The limit of a fiber is necessarily supported on a fiber, say over \(t \in \pp^1\). If \(f_0^{-1}(t)\) does not contain a line, there is nothing to prove. Otherwise, the limit of a fiber contains the points on \(C_{k-1}\) over \(t = t_i\) that are in the smooth locus of \(X\), which accounts for \(k-2\) of the \(k\) points of the limit.
Thus the limit of a fiber must contain exactly \(2\) points on $L_i$.
\end{proof}

\begin{proof}[Proof of Theorem~\ref{thm:main_inductive} parts~\eqref{main_inductive:pushforward} and~\eqref{main_inductive:h1vanishes}]
Applying Lemma~\ref{Lem:2PointsOnALine}, the limit of any fiber of \(f \colon C \to \pp^1\)
spans the corresponding fiber of \(\pp E^\vee\); therefore, the same is true for any fiber of \(f\).
Part  \eqref{main_inductive:pushforward} therefore follows from
 Lemma \ref{good-embeddings}.  
 Next, the bundle \(N_X\) is a positive modification of \(N_{X}^{p,q}\):
 \begin{equation}
 \label{eq:Npq}
 0 \to N_X^{p, q} \hookrightarrow N_X \xrightarrow{\ev_{p,q}} T_{p,q}^1 \to 0.
 \end{equation}
By considering the long exact sequence in cohomology, and using Lemma~\ref{lem:smooth_defs_exist} and the fact that \(T_{p,q}^1\) is punctual, we see that \(H^1(N_X) = 0\).  By semicontinuity, the same is true for a deformation \(C\) of \(X\), which completes the proof of Theorem \ref{thm:main_inductive}\eqref{main_inductive:h1vanishes}.
\end{proof}

\medskip

Let $\Hilb_{k,g}(\pp E^\vee)$ be the Hilbert scheme of curves in $\pp E^\vee$ having arithmetic genus $g$ and relative degree $k$ over $\pp^1$.  Since \(H^1(N_X) = 0\) by Theorem~\ref{thm:main_inductive}\eqref{main_inductive:h1vanishes}, we obtain the following.
 
\begin{cor}
\label{Cor:SmoothPoint}
The curve $X \subset \pp E^\vee$ is a smooth point of $\Hilb_{k,g}(\pp E^\vee)$, and hence lives in a unique component of $\Hilb_{k,g}(\pp E^\vee)$.
\end{cor}

\section{The nodal curve is general}
\label{sec:degeneration_general}

We saw in Corollary~\ref{Cor:SmoothPoint} above that $X$ is contained in a unique component of $\Hilb_{k,g}(\pp E^\vee)$.
In this section, we will establish
Theorem~\ref{thm:main_inductive}\eqref{main_inductive:general}, i.e., we will
show that this component dominates some component of the Hurwitz space $\H_{k,g}$.
(Recall that when the characteristic exceeds \(k\), the Hurwitz space is irreducible and thus this component dominates the entire Hurwitz space.)
Since covers are determined by their ramification data, this can be done by analyzing the ramification
behavior of a general deformation of \(X\).

Recall that, by the Riemann--Hurwitz formula, a simply branched cover
\(C \to \pp^1\)
has \(2g + 2k - 2\) branch points when the characteristic is zero or odd.
However, when the characteristic is \(2\), there is necessarily wild ramification.
The generic behavior is that the cover looks locally in a neighborhood of a ramification point like an Artin-Schreier cover
\[y^2 + y = \frac{\omega}{x},\]
for some constant \(\omega\) (depending on the choice of local coordinate \(x\) at the branch point).
In this case the cover has \(g + k - 1\) branch points.

Our goal is thus to show that a general deformation of \(X\) is simply branched with general branch points,
and moreover, if the characteristic is \(2\), then the corresponding constants \(\omega\) are also general.

By induction, these conditions hold for \(C_{k-1} \to \pp^1\).
We therefore have to check that a general deformation of \(X\)
creates the claimed general branching behavior near each \(t_i\). 
Recall from \eqref{lem:NX_restL} that \(N_X|_{L_i}\) has a quotient \(N_{\pp E_{t_i}^\vee}|_{L_i}(p_i+q_i)\).
Let \(x\) be a local coordinate for \(\pp^1\) at \(t_i\),
inducing an isomorphism
\[N_{\pp E_{t_i}^\vee}|_{L_i}(p_i+q_i) \simeq \O_{L_i}(p_i + q_i).\]
Let \(s\) be a global coordinate on \(L_i\), so that \(p_i\) is at \(s = \infty\) and \(q_i\) is at \(s = 0\).
Then if \(\sigma\) is a section of \(N_X\) whose image in \(H^0(\O_{L_i}(p_i + q_i))\) is \(as + b/s + c\),
the geometry of the first-order deformation of \(X\) corresponding to \(\sigma\) is given by the graph of  the image of \(\sigma\):
\begin{equation} \label{locgeom}
x = (as + b/s + c) \cdot \epsilon.
\end{equation}

\begin{lem} \label{lem:deform_p_eqs}
The branch points of \eqref{locgeom} occur at
\[x = (c \pm 2\sqrt{ab}) \cdot \epsilon.\]
Furthermore, if the characteristic is \(2\), then after a change of coordinates,
\eqref{locgeom} becomes:
\[y^2 + y = \frac{\sqrt{ab} \cdot \epsilon}{x - c\epsilon}.\]
\end{lem}
\begin{proof}
The ramification points satisfy \(dx/ds = (a - b/s^2) \cdot \epsilon = 0\).
Solving for \(s\), we find \(s = \pm \sqrt{b/a}\); substituting this into \eqref{locgeom},
we conclude that \eqref{locgeom} is branched at the claimed values of \(x\).

Note that, when the characteristic is \(2\), there is only one branch point, at \(x = c\epsilon\).
To obtain the desired equation, we
make the change of coordinates:
\[s = \sqrt{\frac{b}{a}} \cdot \left(1 + \frac{1}{y}\right). \qedhere\]
\end{proof}

The key point is that the branching behavior of the corresponding deformation of \(X\) is determined
by the products \(ab\).
(Since the \(t_i\) were already general, the value of \(c\) can be absorbed into a shift of the \(t_i\).)
The smoothing parameters \(a\) and \(b\) appearing in \eqref{locgeom} are the images of the section \(\sigma\) in the deformation space \(T_{p_i}^1 \oplus T_{q_i}^1\) of the singularities at \(p_i\) and \(q_i\).
The parameter \(ab\) appearing in Lemma~\ref{lem:deform_p_eqs} is therefore the image of \(\sigma\) under the quadratic map
\begin{equation} \label{quadmap}
H^0(N_X) \to \bigoplus_{i=1}^m T_{p_i}^1 \oplus T_{q_i}^1 \to \bigoplus_{i=1}^m T_{p_i}^1 \otimes T_{q_i}^1,
\end{equation}
which multiplies the smoothing parameters of the two nodes \(p_i\) and \(q_i\) together.
Our goal is thus to show the surjectivity of the quadratic map \eqref{quadmap}.

The first step is to reduce this problem to the surjectivity of a linear map.
By \eqref{eq:Npq}, together with Lemma~\ref{lem:smooth_defs_exist}, there is a section $s \in H^0(N_X)$ whose image under \eqref{quadmap} is nonzero in each component on the right-hand side. Now suppose
\[s' \in H^0(N_X|_{S_k \cup L_1 \cup \cdots \cup L_m}(-\Gamma) ) = \ker[H^0(N_X) \to H^0(N_X|_{C_{k-1}})] \subset H^0(N_X).\]
Then $s' + s$ has fixed nonzero evaluation in each factor $T_{p_i}^1$ (independent of choice of $s'$).
To show that \eqref{quadmap} is surjective, it therefore suffices to show that such sections $s'$ can attain any collection of values in the $T_{q_i}^1$ factors, i.e., to establish the surjectivity of the linear map
\[  
H^0(N_X \vert_{S_k \cup L_1 \cup \cdots \cup L_m}(-\Gamma))  \to \bigoplus_{i=1}^m T_{q_i}^1.
\]

By \eqref{lem:NX_restL}, we have $H^1(N_X \vert_{L_i}(-p_i - q_i)) =0$, so considering the long exact sequence in cohomology associated to \eqref{lns}, we see that 
\(N_X \vert_{S_k \cup L_1 \cup \cdots \cup L_m}(-\Gamma) \to N_X \vert_{S_k}\)
is surjective on global sections.
It thus suffices to show that $H^0(N_X \vert_{S_k}) \to \bigoplus_{i=1}^m T_{q_i}^1$ is surjective.  To see this, recall that we have an exact sequence
\[
0 \rightarrow N_{S_k} \rightarrow N_X \vert_{S_k} \rightarrow \bigoplus_{i=1}^m T_{q_i}^1 \rightarrow 0,
\]
and $H^1(N_{S_k}) = 0$ by (\ref{eq:h1NS}).

\section{Intersections with subscrolls}
\label{Sec:Subscroll}

In this section, we prove Theorem~\ref{thm:main_inductive}\eqref{main_inductive:subscroll}.  
To do so, we first prove the following general fact about evaluation maps of line bundles. 

\begin{lem}
\label{Lem:Eval}
Let $C$ be a smooth curve, and $\L$ a line bundle on $C$, and $D$ an effective divisor on $C$, and $\Delta \subset \L|_D$ a proper linear subspace.  Let $\Gamma$ be a general set of points on $C$ with $\# \Gamma \geq h^1 (\L) + 1$.
Then the image of the evaluation map
\[
H^0 (\L (\Gamma)) \to \L|_D
\]
is not contained in $\Delta$.
\end{lem}

\begin{proof}
Let $p \in C$ be a general point.  If $h^0 (\L) \neq 0$ and the statement holds for $\L (-p)$, then it holds for $\L$ as well.  Similarly, if $h^1 (\L) \neq 0$ and the statement holds for $\L(p)$, then it holds for $\L$ as well.  We may therefore reduce to the case where $h^0 (\L) = h^1 (\L) = 0$.

Let $A = a_1 + \cdots + a_n$ be a general effective divisor on $C$ of the same degree as $D$.  We have an exact sequence
\[
 H^0 (\L (A)) \to \L|_D \to H^1 (\L(A - D )) .
\]
We claim that $h^1 (\L(A - D )) = 0$. Indeed, specializing to the case $A = D$, we have $h^1(\L) = 0$, so since $A$ is general, the claim follows by semicontinuity.

In particular, the evaluation map $ H^0 (\L (A)) \to \L|_D $ is surjective.  For each $i$, there is a nonzero section of $H^0 (\L(A))$ that vanishes on $A -a_i$, and these sections span $H^0 (\L (A))$.  It follows that there exists an $i$ such that the image of $H^0 (\L(a_i))$ is not contained in $\Delta$.
\end{proof}

\begin{proof}[Proof of Theorem~\ref{thm:main_inductive}\eqref{main_inductive:subscroll}]
By induction on $k$, the curve $C_{k-1}$ meets $\Sigma_{k-2}$ at finitely many points, and does not meet $\Sigma_{k-3}$.  Let $p \in C_{k-1}$ be one of the finitely many points in the intersection $C_{k-1} \cap \Sigma_{k-2}$.  We show that there is a deformation of \(X\) that does not meet $\Sigma_{k-2}$ in a neighborhood of $p$.  Since both \(\Sigma_{k-2}\) and $C_{k-1}$ are contained in \(\Sigma_{k-1}\), it suffices to show that there exists a section of \(N_X\) whose image in \(N_{\Sigma_{k-1}}|_p\) is nonzero. 
By Lemma~\ref{Cor:Surj}, the map
\[
H^0 (N_X) \to H^0 (N_{\Sigma_{k-1}} \vert_{C_{k-1}}(\Gamma))
\]
is surjective.  Now consider the evaluation map 
\begin{equation}\label{eq:eval_p}
H^0 (N_{\Sigma_{k-1}} \vert_{C_{k-1}}(\Gamma)) \to N_{\Sigma_{k-1}}|_p.
\end{equation}
By construction,
\begin{align*}
\#\Gamma = m &\geq 1 + \sum_{i < k} \max \{ 0, e_k - e_i - 1 \} && \text{ by \eqref{mdef}}  \\
&= 1 + h^1(N_{\Sigma_{k-1}} \vert_{C_{k-1}})  && \text{ by \eqref{eq:h1_NF}}.
\end{align*}
Hence, applying Lemma \ref{Lem:Eval}, we see that \eqref{eq:eval_p} is nonzero.  The preimage of the zero subspace of \(N_{\Sigma_{k-1}}|_p\) in \(H^0(N_X)\) is therefore a hyperplane, and a general section of \(N_X\) separates $C_{k-1}$ and \(\Sigma_{k-2}\).
This shows that a general deformation of $X$ does not meet $\Sigma_{k-2}$.

Since a general deformation of $X$ is irreducible and $X$ is not contained in $\Sigma_{k-1}$, a general deformation of $X$ meets $\Sigma_{k-1}$ in finitely many points.
\end{proof}

\section{``Relative" ampleness: maps to the nonnegative scroll}

In this section, we prove the necessary and sufficient conditions for (birational) relative very ampleness. We start by proving the statements in the introduction, and then give the full statement in Section \ref{stronger}.

\begin{proof}[Proof of Theorem \ref{bi-rel-va-intro}]
The dimension of the span of a fibral divisor under $C \to \pp^r$ is the dimension of its span in $C \to \pp E^\vee \to \pp \left( E^\vee / F_-\right)$. It is therefore necessary that the fibers of $\pp \left( E^\vee / F_-\right) \to \pp^1$ have dimension at least $p$, equivalently that $e_{k-p} \geq 0$.

To prove that $e_{k-p} \geq 0$ is sufficient, we induct on \(k\).  We assume the statement for all \(k' < k\).  Let 
\[X = C_{k-1} \cup L_1 \cup \cdots L_m \cup S_k\]
be the degenerate curve constructed in Section~\ref{sec:our_degeneration}, and let \(C\) be a general deformation of \(X\) in \(\pp E^\vee\).

Assume that \(\vec{e}\) has at least \(p+1\) nonnegative parts. Let $(f,\L) \in \we$ be general. We will show that only finitely many fibral divisors \(D\) of degree \(p+1\) satisfy $h^0 (C,\L(-D)) \geq h^0 (C,\L) - p$.  
Let \(x\) be a general point on \(C\).  Every component of \(C \times_{\pp^1} \cdots \times_{\pp^1} C\) dominates \(C\), and hence contains a divisor whose support contains \(x\).
By upper-semicontinuity, it therefore suffices to show \(h^0 (C,\L(-D)) = h^0(C, \L) - (p+1)\) for every divisor \(D\) of the form \(D' + x\), where \(D'\) is supported in the same fiber as \(x\).

To establish this, let $x_0$ be a general point on the section $S_k$, and suppose $x$ specializes to $x_0$.  Every point of $X$ distinct from $x_0$ in the same fiber lies on the component $C_{k-1}$, and by induction, the restriction of $\L$ to $C_{k-1}$ is birationally relatively $(p-1)$-very ample.  Since \(x_0\) is general, it follows that all divisors $D'$ of degree $p$ on $C_{k-1}$ in the fiber \(f_0^{-1}(f_0(x_0))\) have linear span in \(\pp \left( F_{k-1} / F_-\right) \) that is $(p-1)$-dimensional.  Since \(x_0\) does not lie in $\pp \left( F_{k-1} / F_-\right) $, the linear span of $D'+x_0$ in \(\pp \left( E^\vee / F_-\right) \) is $p$-dimensional, and the result follows.
\end{proof}

The remainder of this subsection is devoted to proving Theorem \ref{rel-va-intro}.
If \( \vec{e}\) has at least \(p+2\) nonnegative parts, then \(\vec{e}_{<k}\) has at least \(p+1\) nonnegative parts.  Therefore \(\L|_{C_{k-1}}\) is birationally relatively \(p\)-very ample by Theorem~\ref{bi-rel-va-intro}.  There are finitely many ``bad'' fibers containing either a line \(L_i\) or a fibral divisor \(D'\) of degree $p+1$ on \(C_{k-1}\) with \(h^0(C_{k-1}, \L(-D')) \geq h^0(C_{k-1}, \L) - p\).  Any fibral divisor \(D\) of degree $p+1$ on \(X\) that is not supported in a bad fiber has linear span in \(\pp \left( E^\vee / F_{-} \right)\) of dimension \(p\).  It therefore suffices to show that no divisor of degree \(p+1\) supported in a bad fiber is a limit of a divisor from \(C\) failing to impose independent conditions on \(\L\).  We split this into two cases, first considering the fibers that contain a line \(L_i\).

\begin{lem}
\label{Lem:FibralSupport}
Suppose that $e_{k-p-1} \geq 0$.
Let $D$ be the limit on $X$ of a fibral divisor of degree $p+1$ on $C$ whose span in $\pp \left( E^\vee / F_{-} \right)$ has dimension $p-1$ or less.  Then $D$ is supported on $C_{k-1}$ in a fiber not containing a line and has span of dimension \(p-1\).
\end{lem}

\begin{proof}
Because $D$ contains at most two points on a line by Lemma~\ref{Lem:2PointsOnALine}, the span of $D$ is equal to the span of $D' + x_0$ where \(x_0\) is some point on \(X\) and $D'$ is a degree $p$ divisor on $C_{k-1}$ (if two points limit to a line $L_i$, replacing one of them with $p_i = C_{k-1} \cap L_i$ does not change the span and brings our divisor to this form). 
By the inductive hypothesis, the span of $D'$ in  $\pp(F_{k-1}/F_-)$ has dimension $p-1$. Thus, $x_0$ must already lie in the span of $D'$. Hence, $D$ is supported on $C_{k-1}$. 

By Theorem \ref{bi-rel-va-intro}, a general fibral divisor of degree \(p+1\) supported on \(C_{k-1}\) has span in \(\pp (F_{k-1}/F_-)\) of dimension \(p\).  Because the lines are attached in general fibers, $D$ lies in a fiber not containing a line.
\end{proof}

\begin{proof}[Proof of Theorem~\ref{rel-va-intro}]

Let $D$ be the limit on $X$ of a fibral divisor of degree $p+1$ on $C$ whose span in $\pp \left( E^\vee / F_{-} \right)$ has dimension $p-1$ or less.  By Lemma~\ref{Lem:FibralSupport}, we may assume that $D$ is contained in a fiber of $C_{k-1}$ and has a span of dimension $p-1$. By Theorem \ref{bi-rel-va-intro}, there are finitely many such divisors, so it suffices to show that for each one, we can find a deformation of $X$ so that any corresponding deformation of $D$ has image in $\pp(E^\vee/F_-)$ of dimension $p$.

Let
$V \subseteq H^0(N_X) $ be the subspace of sections that preserve the property that some corresponding deformation of \(D\) lives in a linear space of dimension $p-1$ in $\pp(E^\vee/F_-)$.
We shall show that the image of $V$ under
\begin{equation} \label{123} H^0(N_{X}) \to H^0(N_{\Sigma_{k-1}}|_{C_{k-1}}(\Gamma)) \to N_{\Sigma_{k-1}}|_D 
\end{equation}
is contained in a proper linear subspace $\Delta \subset N_{\Sigma_{k-1}}|_D$. 
Since the first map is surjective (Lemma \ref{Cor:Surj}), it will then suffice to show that there is a section of $N_{\Sigma_{k-1}}|_{C_{k-1}}(\Gamma)$ whose image under the second map does not lie in $\Delta$ (for which we will use Lemma \ref{Lem:Eval}).

Because $D$ does not meet $\pp F_-$ (by Theorem \ref{thm:main_inductive}\eqref{main_inductive:subscroll}), we know $N_{\Sigma_{k-1}}|_D \simeq  N_{\pp(F_{k-1}/F_-)/\pp(E^\vee/F_-)}|_{D}$. 
Let $\Lambda \subset \pp(E^\vee/F_-)$ be the span of the image of $D$ under projection from $\pp F_-$.
If a deformation of $D$ continues to have span in $\pp(E^\vee/F_-)$ of dimension $p-1$, then there would exist a deformation of $\Lambda$ that contains it. Therefore 
the image of $V$ under the composition \eqref{123} is necessarily contained in the image
$\Delta$ of
\begin{equation} \label{s} H^0(N_{\pp(F_{k-1}/F_-)/\pp(E^\vee/F_-)}|_{\Lambda}) \rightarrow N_{\pp(F_{k-1}/F_-)/\pp(E^\vee/F_-)}|_{D} = N_{\Sigma_{k-1}}|_D.
\end{equation}

Since $\Lambda$ is contained in a fiber and $\pp(F_{k-1}/F_-) \subset \pp(E^\vee/F_-)$ is a hyperplane in each fiber, we have $N_{\pp(F_{k-1}/F_-)/\pp(E^\vee/F_-)}|_{\Lambda} \simeq \O_{\Lambda}(1)$. Hence, the source of \eqref{s} has dimension
\[h^0(N_{\pp(F_{k-1}/F_-)/\pp(E^\vee/F_-)}|_{\Lambda}) = \dim \Lambda + 1 = p.\]
On the other hand, $D$ has degree $p+1$, so the target of \eqref{s} has dimension $p+1$. Hence, the image $\Delta \subset N_{\Sigma_{k-1}}|_D$ of \eqref{s} is a proper subspace. 
Finally,
by Lemma \ref{Lem:Eval}, since $\Gamma$ is general and contains $m \geq h^1(N_{\Sigma_{k-1}}|_{C_{k-1}})+1$ points, there exists a section in
$H^0(N_{\Sigma_{k-1}}|_{C_{k-1}}(\Gamma))$ that misses
$\Delta \subset N_{\Sigma_{k-1}}|_D$.
\end{proof}

\subsection{Necessary and sufficient conditions for relative very ampleness} \label{stronger}

The following theorem is a stronger version of Theorem \ref{rel-va-intro} from the introduction.
\begin{thm} \label{rel-va} 
Assume $p \leq k - 1$.
A general line bundle in  \(\we\) is relatively \(p\)-very ample if and only if either
  \begin{enumerate}
\item\label{rel-va-1} $e_{k-p-1} \geq 0$, 
\item\label{rel-va-2} $p = 0$ and $e_k \geq 0$ and $e_{k-1} - e_1 \leq 1$ and $\rho'(g,\vec{e}) = 0$, 
\item\label{rel-va-3} $p = k-2$ and $e_{2} \geq 0$ and $e_k - e_2 \leq 1$ and $\rho'(g,\vec{e}) = 0$, 
\item\label{rel-va-4} $p = k-1$ and $e_1 \geq 0$, or
\item\label{rel-va-5} $g = 0$ and $e_{k-p} \geq 0$.
\end{enumerate}
\end{thm}

By Theorem \ref{bi-rel-va-intro}, the condition $e_{k-p} \geq 0$ is necessary, and by 
Theorem~\ref{rel-va-intro}, the condition $e_{k-p-1} \geq 0$ is sufficient.
We therefore suppose $e_{k-p-1} < 0 \leq e_{k-p}$.

By Theorem \ref{bi-rel-va-intro}, there are finitely many fibral divisors of degree $p+1$ that fail to be independent in their fiber of $\pp(E^\vee/F_-)$ over $\pp^1$. We shall call these \defi{dependent fibral divisors}.
Below, we use intersection theory to compute the number $N$ of dependent fibral divisors of degree $p+1$, counted with multiplicity. In particular, when $N = 0$, there are no dependent fibral divisors of degree $p+1$, so a general $\L$ will be relatively $p$-very ample (instead of just birationally relatively $p$-very ample). In other words, in these cases, $\L$ is more ample than expected because of a numerical coincidence.

On the other hand, of course, if $N \neq 0$, then $\L$ cannot be relatively $p$-very ample. This will prove the only if direction in Theorem \ref{rel-va}.

\subsubsection{Number of dependent fibral divisors via intersection theory}
To perform our intersection theory calculation, we realize the locus of dependent fibral divisors of degree $p+1$ as a degeneracy locus where a map of vector bundles drops rank. Let $A$ be the closure of the complement of the diagonals in the $(p+1)$-fold fiber product $C \times_{\pp^1} \times \cdots \times_{\pp^1} C$. Over a point $t \in \pp^1$ that is not a branch point of $f$,
the fiber of $h\colon A \to \pp^1$ consists of ordered tuples $(a_1, \ldots, a_{p+1})$ of $p+1$ distinct points in the fiber of $f$ over $t$.
The map $h\colon A \to \pp^1$ is therefore finite of degree $k!/(k - p-1)!$. 

If $t \in \pp^1$ is a branch point of $f$, the fiber of $h\colon A \to \pp^1$ over $t$ consists of tuples $(a_1, \ldots, a_{p+1})$ where $a_i = a_j$ is a ramification point, for some pair $(i,j)$. Let $Z \subset A$ be the locus of points $(a_1, \ldots, a_{p+1}) \in A$ where the $a_i$ fail to be distinct. By Riemann--Hurwitz, the degree of the ramification divisor of $f$ is $2g -2 + 2k$. It follows that
\[\deg Z = (2g - 2 + 2k)\cdot {p+1 \choose 2} \cdot \frac{(k-2)!}{(k - p - 1)!}. \]

Let $\pi_j\colon A \to C$ be the projection onto the $j$th factor. Note that $\deg \pi_j = (k-1)!/(k-p-1)!$.
For each $j$ there is an evaluation map
\begin{equation} \label{evh} h^*E = \pi_j^*f^*(f_*\L) \to \pi_j^*\L.
\end{equation}
The fiber of the kernel of \eqref{evh} over a point $(a_1, \ldots, a_{p+1}) \in A$ is the space of linear forms on the fiber of $\pp E^\vee$ that vanish on $a_j \in C \subset \pp E^\vee$.
Taking the sum of \eqref{evh} over all $j$, we obtain a map of vector bundles
\[h^*E \rightarrow \bigoplus_{j=1}^{p+1} \pi_j^*\L \]
on $A$ whose kernel is the bundle of linear forms on the fibers of $\pp E^\vee$ that vanish on $a_1, \ldots, a_{p+1}$.

Recall that we assume $E$ has exactly $p+1$ non-negative parts, which form a canonical rank $p+1$ subbundle. (This subbundle of $E$ is the same as $(E^\vee/F_-)^\vee$, which is the space of linear forms on our projective bundle $\pp(E^\vee/F_-)$.)
 The composition
\[\phi\colon h^*\left(\bigoplus_{i \geq k-p} \O(e_i) \right) \rightarrow h^*E \to \bigoplus_{j=1}^{p+1}\pi_j^*\L\]
drops rank when there exists a linear form on $\pp(E^\vee/F_-)$ that vanishes on all of $a_1, \ldots, a_{p+1}$. In other words, $\det \phi$ vanishes when $a_1, \ldots, a_{p+1}$ are linearly dependent in $\pp(E^\vee/F_-)$. If $a_i = a_j$ (so $(a_1, \ldots, a_{p+1}) \in Z$), then $a_1, \ldots, a_{p+1}$ are automatically dependent. Hence, $Z \subset V(\det \phi)$. Meanwhile, away from $Z$, the determinant $\det \phi$ vanishes precisely when 
$a_1 + \ldots + a_{p+1}$ is a dependent fibral divisor. By Theorem \ref{bi-rel-va-intro}, there are finitely many dependent fibral divisors, so $\det \phi$ vanishes in the expected codimension.
Note also that Theorem~\ref{rel-va-intro} ensures a general $(f, \L) \in \W^{\vec{e}}_{\text{BN}}$ is relatively $(p+1)$-very ample, so there are no dependent fibral divisors of degree $p$.
Thus, the degree of the scheme of dependent fibral divisors of degree $p+1$ is
\begin{align}
N &= \frac{1}{(p+1)!}(\deg [V(\det \phi)] - \deg Z) \\
& = \frac{1}{(p+1)!} \left( \sum_{j=1}^{p+1} \deg \pi_j^*\L - \sum_{i = k-p}^k \deg h^*\O(e_i) - \deg Z \right) \notag \\
&=\frac{1}{(p+1)!} \left( (p+1) (\deg \pi_i)(\deg \L) - (\deg h)(e_{k-p} + \ldots + e_k) - \deg Z \right) \notag \\
& = {k-1 \choose p} (e_1 + \ldots + e_k + g + k - 1) - {k \choose p+1}(e_{k-p} + \ldots + e_k) - (g - 1 + k){k-2 \choose p-1}. \label{hi}
\end{align}

\subsubsection{The case of no dependent fibral divisors}
The edge cases when $\L$ is more ample than expected occur when $N =0$. Multiplying \eqref{hi} by $\frac{(p+1)!(k - 1 - p)!}{(k-2)!}$, the condition $N = 0$ is equivalent to
\[
(p+1)(k-1)(e_1 + \ldots + e_k ) + (p+1)(k-1 - p)(g + k - 1) = k(k-1)(e_{k-p} + \ldots + e_k),
\]
or equivalently,
\begin{align*}
(p+1)(k-1 - p)(g + k - 1) &= (k-1)\left((k - p - 1)(e_{k-p} + \ldots + e_k) - (p+1)(e_1 + \ldots + e_{k-p-1})\right) \\
&=(k-1)\sum_{i = 1}^{k-p-1} \left(\sum_{j=k-p}^{k} e_j - e_i\right).
\end{align*}
This in turn means
\begin{equation} \label{meq}
(p+1)(k-1-p)g = (k-1)\sum_{i = 1}^{k-p-1} \left(\sum_{i=k-p}^{k} e_j - e_i - 1\right) \leq (k - 1) u(\vec{e}) \leq (k - 1)g.
\end{equation}
If $g = 0$, then all inequalities in \eqref{meq} are equalities and hence $N = 0$. Therefore, we have relative $p$-very ampleness whenever there are $p+1$ nonnegative parts, which is Theorem \ref{rel-va}\eqref{rel-va-5}.

Assume for the remainder that $g \neq 0$.
Dividing \eqref{meq} by $g$, we see that
\begin{equation} \label{last}
(p+1)(k-1-p) \leq k - 1. 
\end{equation}
There are only three possible values for $p$ where this inequality holds, which translate into Theorem~\ref{rel-va}\eqref{rel-va-2}--\eqref{rel-va-4}.
\begin{itemize}
\item $p = 0$. Since equality holds in \eqref{last}, all inequalities in \eqref{meq} are actually equalities. Hence, $g = u(\vec{e})$, equivalently $\rho'(g,\vec{e}) =0$. Furthermore, $\sum_{i = 1}^{k-1} e_k - e_i - 1= u(\vec{e})$, which implies that the other parts are balanced, i.e., $e_{k-1} - e_1 \leq 1$.
\item $p = k-2$. Again, equality holds in \eqref{last}, so all inequalities in \eqref{meq} are actually equalities. This gives $g = u(\vec{e})$, equivalently $\rho'(g,\vec{e}) = 0$. Furthermore, $\sum_{i=2}^{k} e_j - e_1 - 1 = u(\vec{e})$, which implies that the other parts are balanced, i.e., $e_k - e_2 \leq 1$.
\item $p=k-1$. Our assumption that $\vec{e}$ has exactly $p+1$ nonnegative parts means $e_1 \geq 0$.
\end{itemize}

\section{Very ampleness: map to projective space\label{Sec:VA}}
In this section, we use the complete characterization in Theorem~\ref{rel-va} of relative \(p\)-very ampleness to prove Theorem~\ref{Thm:bpf} and a refined version of Theorem \ref{Thm:VA-intro} giving necessary and sufficient conditions for very ampleness.

\begin{proof}[{Proof of Theorem~\ref{Thm:bpf}}]
When \(p = 0\),  relative \(0\)-very ampleness is equivalent to \(0\)-very ampleness.  It therefore suffices to show that conditions \eqref{rel-va-1}--\eqref{rel-va-5} in Theorem~\ref{rel-va} are equivalent to Theorem~\ref{Thm:bpf}.  Conditions~\eqref{rel-va-1}--\eqref{rel-va-5} in Theorem~\ref{rel-va} when \(p=0\) simplify to
\begin{enumerate}[(a)]
\item \label{a} $e_{k-1} \geq 0$, or
\item \label{b} $e_k \geq 0$ and $e_{k-1} - e_1 \leq 1$ and $\rho'(g, \vec{e}) = 0$.
\end{enumerate}
(To show that Theorem~\ref{rel-va}\eqref{rel-va-5}  reduces to \eqref{b} above, note that the pushforward of a line bundle under any map of genus \(0\) curves is balanced.)

Finally, suppose that \eqref{b} holds and \eqref{a} does not, so $e_{k-1} < 0$.
In this case, we have
\[h^0(\L\otimes f^*\O_{\pp^1}(-e_k)) = 1 \qquad \text{and} \qquad h^1(\L \otimes f^*\O_{\pp^1}(-e_k)) = \sum_{i < k} e_k - e_i - 1 = u(\vec{e}) = g. \]
Hence, $\L \otimes f^*\O_{\pp^1}(-e_k) \cong \O_C$, or equivalently, $\L \cong f^*\O_{\pp^1}(e_k)$.
\end{proof}

Ordinary \(p\)-very ampleness
implies relative \(p\)-very ampleness, but the converse implication
does not hold for $p \geq 1$. Nevertheless, we have the following inductive tool:

\begin{lem} \label{lem:it}
For a general $(f, \L) \in \we$, the line bundle $\L$ is \(p\)-very ample
(resp.\ birationally \(p\)-very ample)
if both:
\begin{itemize}
\item It is relatively \(p\)-very ample (resp.\ birationally relatively \(p\)-very ample), and
\item For a general \((f',\L') \in \W^{\vec{e}'}_{\mathrm{BN}}\)
where
\(\vec{e}' = (e_1 - 1, e_2 - 1, \ldots, e_k - 1)\), the line bundle $\L'$ is \((p - 1)\)-very ample
(resp.\ birationally \((p - 1)\)-very ample).
\end{itemize}
\end{lem}

\begin{proof}
Take any effective divisor \(D\) of degree \(p + 1\),
and suppose \(D\) nontrivially intersects a fiber \(F\) of the map
\(C \to \pp^1\).
Because \(\L\) is relatively \(p\)-very ample,
\(D \cap F\) is in linear general position.
It thus suffices to show that the projection from \(F\)
of \(D - (D \cap F)\) is in linear general position.
This follows from the \((p - 1)\)-very ampleness
of \(\L' = \L \otimes f^* \O_{\pp^1}(-1)\).
\end{proof}

\begin{rem} \label{rem:dumb}
As an immediate consequence of Theorem~\ref{Thm:bpf} and Lemma~\ref{lem:it},
we see that
for a general $(f, \L) \in \we$, the line bundle $\L$ is $p$-very ample if:
\begin{align*}
e_{k} & \geq p \\
e_{k-1} &\geq p \\
e_{k-2} &\geq p-1 \\
\vdots  \\
e_{k-p-1} &\geq 0.
\end{align*}
Similarly, for a general \((f, \L) \in \we\), the line bundle $\L$ is birationally \(p\)-very ample if:
\begin{align*}
e_k &\geq p \\
e_{k-1} &\geq p-1 \\
\vdots  \\
e_{k-p} &\geq 0.
\end{align*}
\end{rem}

The result in Remark~\ref{rem:dumb} is likely far from sharp.
Farkas proved \cite{farkas} that a general line bundle in the classical Brill--Noether locus \(\mathcal{W}^r_d\) is \(p\)-very ample if \(r \geq 2p+1\).  We have already shown that a general line bundle in \(\we\) is \textit{relatively} \(p\)-very ample if \(e_{k-p-1}\geq 0\), and so the natural conjecture would be as follows.

\begin{conj}\label{p-very-conj}
A general line bundle in $\we$ is  \(p\)-very ample if 
\[e_{k-p-1} \geq 0 \quad \text{and} \quad r = h^0(\O(\vec{e})) - 1 \geq 2p+1.\]
\end{conj}

When \(p=1\), we prove this conjecture in Theorem~\ref{Thm:VA-intro}; in other words, the remainder of the paper deals with the discrepancy between what we expect by Conjecture~\ref{p-very-conj} and what is covered by Remark~\ref{rem:dumb} when \(p=1\).  Note that this discrepancy increases as \(p\) increases.

\medskip

We now give a characterization of birational very ampleness, which will feed into our proof of Theorem \ref{Thm:VA-intro}. 

\begin{lem} \label{lem:3parts}
For $(f, \L) \in \we$ general, $\L$ is birationally very ample if and only if
\begin{enumerate}
\item \label{pt1} \(e_{k-2} \geq 0\), 
\item \(e_{k-1} \geq 0\) and \(e_{k}\geq 1\), or
\item \(g=0\) and \(e_{k-1} \geq 0\).
\end{enumerate}
\end{lem}
\begin{proof}
By Theorem~\ref{bi-rel-va-intro}, it is necessary that \(e_{k-1} \geq 0\).  If in addition \(e_{k} \geq 1\), then  a generic line bundle in \(\W^{\vec{e}}_{\text{BN}}\) is birationally very ample by Remark \ref{rem:dumb}. On the other hand, if \(e_{k-1} = e_k = 0\) and \(e_{k-2} < 0\), then \(|\mathcal{L}|\) maps \(C\) to \( \pp^1\), and so \(\L\) is birationally very ample if and only if \(g=0\).

It remains to consider the case \(e_{k-2} \geq 0\).  We will use our degeneration introduced in Section~\ref{sec:our_degeneration}.
The map $X \to \pp^r$ sends $C_{k-1}$ into a proper linear space (with non-zero degree) and sends $S_k$ to a rational normal curve in a complementary linear space (possibly contracting it to a point if \(e_k = 0\)). Since the lines $L_i$ are attached at general points of $C_{k-1}$, their images in $\pp^r$ are distinct. In particular, the map $X \to \pp^r$ is birationally very ample along the lines $L_i$.
\end{proof}

\subsection{Proof of Theorem \ref{Thm:VA-intro}} \label{mc}

We must prove very ampleness in the following two cases:
\begin{enumerate}
\item When \(\vec{e}\) has at least four nonnegative parts, none of which are positive.
\item When \(\vec{e}\) has at least three nonnegative parts, at least one of which is positive.
\end{enumerate}

\subsubsection{At least four parts of degree $0$ and no positive parts} \label{edge04}
Here, we consider splitting types with $e_k = e_{k-1} = e_{k-2}= e_{k-3} = 0$. 
We argue by degeneration to the reducible curve $X$ constructed in Section \ref{sec:our_degeneration}.
By our hypothesis on $\vec{e}$, we know $\O_{\pp E^\vee}(1)|_{C_{k-1}}$ is a line bundle whose pushforward has splitting type with at least three parts of degree $0$ (and no positive parts).
Thus, by Lemma~\ref{lem:3parts}\eqref{pt1}, we have that
 $\O_{\pp E^\vee}(1)|_{C_{k-1}}$  is birationally very ample on $C_{k-1}$. 
 The complete linear system of $\O_{\pp E^\vee}(1)$ maps $X$ to $\pp^r$ with $r \geq 3$.
This map contracts $S_k$ to a point and sends
 $C_{k-1}$ birationally onto its image in a complementary hyperplane $\alpha\colon C_{k-1} \to \pp^{r-1} \subset \pp^r$.

We will show that for a deformation $C$ of $X$, every degree \(2\) effective divisor $D$ on $C$ embeds into $\pp^r$. We have two cases to consider:
\begin{enumerate}
\item \label{ck} $D$ limits to a divisor on $C_{k-1}$ which is collapsed to a point in $\pp^{r-1} \subset \pp^r$.
\item \label{sk} $D$ limits to a divisor on $S_{k}$ (which is necessarily collapsed to a point in $\pp^r$).
\end{enumerate}

Case \eqref{ck}: 
By inducting on the number of nonnegative parts of $\vec{e}$,
 we may assume $\vec{e}$ has exactly four parts of degree $0$, i.e., $r = 3$. 
By Lemma \ref{lem:3parts},
there are finitely many bad degree \(2\) effective divisors $D$ on $C_{k-1}$ that are collapsed to a point $\alpha(D) \in \pp^2 \subset \pp^3$, so it suffices to show that a general deformation of $X$ will separate any given bad $D$.
Let $M = N_{\Sigma_{k-1}}|_{C_{k-1}} = \O_{\pp E^\vee}(1)|_{C_{k-1}}$ (see \eqref{eq:normal_bundle_sigma} in case $e_k = 0$).
We have an identification
\[M|_D \simeq \alpha^*N_{\pp^2/\pp^3}|_{\alpha(D)}. \]
A deformation of $X$ separates $D$ in the map to $\pp^3$
if it does not take values in the subspace 
\[\Delta = \alpha^*H^0(N_{\pp^2/\pp^3}|_{\alpha(D)}) \subset H^0(\alpha^*N_{\pp^2/\pp^3}|_{\alpha(D)}) = M|_D. \]
Thus it suffices to show that the image of
\[H^0(N_X) \rightarrow H^0(M(\Gamma)) \rightarrow M|_{D}\]
does not lie in the subspace $\Delta$. By Lemma~\ref{Cor:Surj}, the first map above is surjective. Since $\Gamma$ is general with $\# \Gamma \geq h^1( M) + 1$, the image of $H^0(M(\Gamma)) \rightarrow M|_{D}$ does not lie in $\Delta$ by Lemma~\ref{Lem:Eval}.

Case \eqref{sk}: Every degree \(2\) divisor $D$ on $S_k$ is collapsed to a point in $\pp^r$. 
By \eqref{nsk}
we have a filtration
\begin{equation} \label{eq:filt}
0 \rightarrow \bigoplus_{i \leq k-4} \O(-e_i) \rightarrow N_{S_k} \rightarrow \O^{\oplus 3} \rightarrow 0.
\end{equation}
We have that $N_X|_{S_k}$ is a positive modification of $N_{S_k}$.  Since $C_{k-1}$ does not meet $\Sigma_{k-4}$ by Theorem \ref{thm:main_inductive}\eqref{main_inductive:subscroll}, these modifications
do not meet the subbundle in \eqref{eq:filt}, i.e., \(k=0\) in \cite[Equation (7)]{interpolation}.
Hence, we obtain a filtration
\begin{equation} \label{filtN}
0 \rightarrow  \bigoplus_{i \leq k-4} \O(-e_i) \rightarrow N_X|_{S_k} \rightarrow Q \rightarrow 0,
\end{equation}
where $Q$ is a positive modification of $\O^{\oplus 3}$ at the $q_i$. For $i \leq k-4$, we have $e_i \leq 0$ by assumption, so $h^1( \bigoplus_{i \leq k-4} \O(-e_i))= 0$. Hence, we have a surjection $H^0(N_X|_{S_k}) \to H^0(Q)$.

Since a general section in $H^0(N_X)$ smooths all the $q_i$, and the images of the lines \(L_i\) are distinct, it suffices to show that it separates any degree
 $2$ effective divisor $D$ in $S_k \smallsetminus \{q_1, \ldots, q_m\}$. For such $D$, we have a natural trivialization of $Q|_D$.
If a section in $H^0(N_X)$ has non-constant image under
\[H^0(N_X)  \to H^0(N_X|_{S_k}) \to H^0(Q) \to Q|_{D},\]
then the corresponding deformation separates $D$.
We know $H^0(N_X) \to H^0(N_X|_{S_k})$ is surjective and $H^0(N_X|_{S_k}) \to H^0(Q)$ is surjective.
Therefore, we wish to show that a general section of $Q$
misses the constant subspace when evaluated in $Q|_{D}$ for all degree $2$ effective divisors $D$ supported on \(S_k \smallsetminus \{q_1, \dots, q_m\}\).

There is a $2$-dimensional family of degree $2$ effective divisors $D$ in $S_k \smallsetminus \{q_1, \ldots, q_m\}$. For each such $D$, we get a ``bad subspace" of $H^0(Q)$ defined as the preimage of the constant sections $\Delta \subset Q|_{D}$ under
 $H^0(Q) \to  Q|_{D}$.
If $m \geq  3$, then we claim that all summands of $Q$ are positive. Indeed, because the $p_i$ are general, and the image of \(C_{k - 1}\) in \(\pp^{r - 1}\) is nondegenerate, the modifications cannot all lie in a proper trivial subbundle. Then, it is a codimension $3$ condition on the space of global sections $H^0(Q)$ to lie in the constant subspace along $D$. The union of a $2$-dimensional family of codimension $3$ ``bad subspaces" of $H^0(Q)$ cannot be all of \(H^0(Q)\), so a general section will not lie in any bad subspace.
If $m = 1$ or $m = 2$, then $N_X|_{S_k}$ has an $\O(1)$ summand, and every section that is non-constant in the $\O(1)$ component is non-constant along every $D$.

\subsubsection{At least three nonnegative parts at least one of which is positive} \label{edge00a}
If $e_{k-1} > 0$, the result follows from Remark~\ref{rem:dumb};
we therefore suppose $e_{k-2} = e_{k-1} = 0$ and $b \colonequals e_k > 0$.

 Let $X = C_{k-1} \cup L_1 \cup \cdots \cup L_m \cup S_k \subset \pp E^\vee$ be our degenerate curve. The complete linear series for $\O_{\pp E^\vee}(1)$ sends $X$ to $\pp^{a+b+1}$ by sending $C_{k-1}$ onto a $\pp^a$ and $S_k$ onto a degree $b$ rational normal curve in a complementary $\pp^{b}$. Write $\alpha \colon C_{k-1} \to \pp^a$ for the restriction of $|\O_{\pp E^\vee}(1)|$ to $C_{k-1}$.
Our goal is to show that for a general deformation $C$ of $X$, every degree $2$ effective divisor on $C$ embeds into $\pp^{a+b+1}$. 
The only degree \(2\) effective divisors on $X$ that are collapsed to a point in $\pp^{a+b+1}$ are divisors on $C_{k-1}$ that live in a fiber of $\alpha \colon C_{k-1} \to \pp^a$.

Let $\Gamma = \{p_1, \ldots, p_m\}$ be the points on $C_{k-1}$ where the lines $L_1, \ldots, L_m$ are attached.
Write
\begin{equation} \label{Ldef}
\L \colonequals N_{\Sigma_{k-1}}|_{C_{k-1}} = \O_{\pp E^\vee}(1)|_{C_{k-1}} \otimes f^* \O_{\pp^1}(-e_{k})  = \alpha^* \O_{\pp^1}(1) \otimes f^*\O_{\pp^1}(-b),
\end{equation}
where the second equality follows from \eqref{eq:normal_bundle_sigma}.
By Lemma \ref{lem:Nx_restC} we have $N_{X}|_{C_{k-1}}/(N_{C_{k-1}/\Sigma_{k-1}}) = \L(\Gamma)$.

Because $b> 0$, the map $\pp(E^\vee/F_{k-3}) \to \pp^{a+b+1}$ is an embedding away from the codimension $1$ subscroll $\pp(F_{k-1}/F_{k-3}) \subset \pp(E^\vee/F_{k-3})$.
If a section of $N_{X}$ has image in $\L(\Gamma)$ that does not vanish along a degree \(2\) effective divisor $D$ on $C_{k-1}$, then the corresponding deformation of $X$ separates $D$.
Since $H^0(N_X) \to H^0(\L(\Gamma))$ is surjective by Lemma~\ref{Cor:Surj}, it suffices to show that, for general $\Gamma$, a general section of $\L(\Gamma)$ has a vanishing locus which contains no degree $2$ effective divisor contained in a fiber of $\alpha$.

If $e_{k-3} = 0$, then by Lemma \ref{lem:3parts}, we know that $\alpha\colon C_{k-1} \to \pp^a$ is birational onto its image. Thus, there are finitely many degree $2$ effective divisors $D$ on $C_{k-1}$ that are collapsed by $\alpha$. For each such $D$, Lemma \ref{Lem:Eval} applied to $\Delta = 0 \subset \L|_D$ shows that a general section of $\L(\Gamma)$ is nonzero on $D$.

For the remainder, we therefore assume $e_{k-3} < 0$, and hence $a = 1$. 
If $\deg \alpha = 1$, then no degree \(2\) effective divisor of $X$ is collapsed in the map to $\pp^{b+2}$. We therefore assume for the remainder of this section that $\deg \alpha \geq 2$. In this case, $\alpha\colon C_{k-1} \to \pp^1$ collapses infinitely many degree $2$ effective divisors $D$, making this case more difficult than the previous paragraph.

Throughout, we write $g' \colonequals u(\vec{e}_{<k})$ for the genus of $C_{k-1}$. 
If $\vec{e} = (-1, \ldots, -1, 0, 0, b)$, then we have
 $g' = u(\vec{e}_{<k}) = u(-1, \ldots, -1, 0, 0) = 0$, and
so $\deg \alpha = g' - 1 + (k-1) - \deg(\vec{e}_{<k}) = 1$.
Thus, we may also assume that $e_1 < -1$, and consequently, the quantity $m$ defined in \eqref{mdef} satisfies
\begin{equation} \label{mgeq3}
m \geq 1+ \sum_{j < k} (e_k - e_j - 1) \geq 3.
\end{equation}

If $m \geq g'$, then $\L(\Gamma)$ is a general line bundle, so it has a section with the desired property. We can therefore assume $m < g'$ for the remainder.
 Since we are assuming $b > 0$, we have 
\[h^0(\L) = h^0(f_*\L) = h^0\left(\O(\vec{e}_{<k}) \otimes \O(-b)\right)= 0.\]
We claim that it suffices to consider the case of minimal $m$, namely $m = h^1(\L) + 1$. Indeed, if $m > h^1(\L) + 1$, let $\Gamma' \subset \Gamma$ be a subset of $h^1(\L) + 1$ points. Then, if we have treated the case of minimal $m$, it follows that a general section in
$H^0(\L(\Gamma')) \subset H^0(\L(\Gamma))$ has the desired property.
We therefore assume $m = h^1(\L) + 1$ for the remainder of this section.
From \eqref{mdef} and our assumptions $e_1 < -1$ and $e_i \leq -1$ for $i \leq k-3$, we see that
\[ m = 1+\sum_{i=1}^{k-1} (b - e_i - 1) > 1+ (k-3)b + 2(b-1) = (k-1)b - 1 \geq (k-2)b.\]

\begin{lem} 
Let $x_1, \ldots, x_{m - (k-2)b}$ be general points on $C_{k-1}$. Then there exists
$\Gamma \in \Sym^{m}C_{k-1}$  such that $h^0(\L(\Gamma)) = 1$, and the unique section of $\L(\Gamma)$ vanishes along the $x_i$ but at no other points in the fibers $\alpha^{-1}(\alpha(x_i))$.
\end{lem}
\begin{proof}
Let $d = \deg \alpha$ and let $\{h_1, \ldots, h_d\}$ be the points of a general fiber of $\alpha$. Define
\[\Gamma\colonequals \{x_1, \ldots, x_{m-(k-2)b}\} \cup \bigcup_{i \leq b} \{f^{-1}(f(h_{i})) \smallsetminus h_{i} \}.\]
Now, using \eqref{Ldef}, we see 
\[\L(\Gamma) = \O\left(\Gamma + h_1 + \ldots + h_d - \sum_{i = 1}^b f^{-1}(f(h_i))\right)  = \O(h_{b+1} + \ldots + h_d + x_1 + \ldots + x_{m-(k-2)b}).\]
In particular, $\L(\Gamma)$ has a section that vanishes along the $x_i$ but at no other points in the fibers $\alpha^{-1}(\alpha(x_i))$.
It remains to see that $h^0(\L(\Gamma)) = 1$.

First, we observe that $h^0(\O(h_{b+1} + \ldots + h_d)) = 1$ since
 $h^0(\O(h_1 + \ldots + h_d)) = 2$ and $\O(h_1 + \ldots + h_d)$ is basepoint free.
By Riemann--Roch, 
\[\chi(\O(h_{b+1} + \ldots + h_d)) = \chi(\O(h_1 + \ldots + h_d)) - b = \chi(\L) + b(k-1) - b = [0 - (m+1)] + b(k-2).\]
Hence, $h^1(\O(h_{b+1} + \ldots + h_d)) = m - b(k-2)$, so twisting up by this many general points kills the $h^1$ without increasing $h^0$.
\end{proof}

We now show that for $\Gamma$ general, \emph{every} zero of a general section of $\L(\Gamma)$ is isolated in its fiber under $\alpha$.
For this, consider the incidence correspondence
\[\Psi = \{(x, \Gamma)  \in C_{k-1} \times \Sym^m C_{k-1}: h^0(\L(\Gamma)(- x)) \neq 0 \}.\]
Because $\chi(\L(\Gamma)(-x)) = 0$, the locus $\Psi \subset C_{k-1} \times \Sym^m C_{k-1}$ is pure codimension $1$.

By construction, if $h^0(\L(\Gamma)) = 1$, then the fiber of $\Psi$ over $\Gamma \in \Sym^m C_{k-1}$ is the vanishing locus of the unique section. 
We have already found one such $\Gamma$ where there exists a point $x$ that is the only point of $\alpha^{-1}(\alpha(x))$ in the fiber of $\Psi \to \Sym^m C_{k-1}$ over $\Gamma$. 
It therefore suffices to show that $\Psi$ 
is irreducible, which is Lemma \ref{irr} below. To prove Lemma \ref{irr}, we first need the following.

\begin{lem} \label{lem:h0andh1}
Suppose $L$ is a line bundle on a curve $C$ with $h^0(L) \geq 2$ and $h^1(L) = 0$. Then, for general $x \in C$, the line bundle $L(x)$
 is birationally very ample.
 \end{lem}
 \begin{proof}
 Consider $P = \{(p, q) \in \Sym^2 C : h^0(L(-p - q)) \geq h^0(L) - 1\}$. Since it is defined by a determinantal condition, $P$ is either empty or pure dimension $1$.
If $(p, q) \notin P$, then $p$ and $q$ already impose independent conditions on $L$, so they impose independent conditions on $L(x)$ for all $x \in C$.
Now fix some component of $P$ and a general $(p, q)$ in that component. 
 It suffices to show that, for general $x \in C$, the points $p$ and $q$ impose independent conditions on $L(x)$.
Since every component of $P$ has dimension $1$, we may assume that $p$ and $q$ are not both basepoints of $L$, so $h^0(L(-p-q)) = h^0(L) - 1$, and therefore $h^1(L(-p-q)) = 1$. Since $x$ is general, we have $h^1(L(x)(-p-q)) = 0$. Hence, $h^0(L(x)(-p-q))= h^0(L(x)) - 2$, as desired.
 \end{proof}

\begin{lem} \label{irr}
If $m < g'$, then the variety $\Psi$ is irreducible.
\end{lem}
\begin{proof}
First we show that no fiber of $\Psi$ over $C_{k-1}$ is the entire space $\Sym^m C_{k-1}$.
Let $g'$ be the genus of $C_{k-1}$. 
Recall that $\chi(\L(\Gamma)(-x)) = 0$ for any $x\in C_{k-1}$.
If $h^0(\L(\Gamma)(-x)) > 0$ for all $\Gamma$, then we would have $h^1(\L(\Gamma)(-x)) > 0$ for all $\Gamma$, so  $h^1(\L(-x)) \geq m + 1$. But this is impossible since we are assuming $h^1(\L) = m-1$. Since $\Psi$ is pure codimension $1$, every component must dominate $C_{k-1}$.

It remains to show that for a general $x \in C_{k-1}$, the fiber of $\Psi$ over $x$ is irreducible. Let $K$ denote the canonical bundle on $C_{k-1}$.
Because $\chi(\L(\Gamma)(-x)) = 0$, we have
\begin{equation}\label{cgs}
h^0(\L(\Gamma)(-x)) = h^1(\L(\Gamma)(-x)) = h^0(K \otimes \L^{-1}(x)(-\Gamma)).
\end{equation}
Note that $h^0(K \otimes \L^{-1}) = h^1(\L) = m - 1 \geq 2$ by \eqref{mgeq3} and $h^1(K \otimes \L^{-1}) = h^0(\L) = 0$. 
Therefore, Lemma \ref{lem:h0andh1} says that $K \otimes \L^{-1}(x)$ is birationally very ample. 
Let $s = h^0(K \otimes \L^{-1}(x)) - 1$.
The cohomology group in \eqref{cgs} is non-zero if and only if the image of $\Gamma$ under the map $C_{k-1} \to \pp^s$ given by $|K \otimes \L^{-1}(x)|$
 is contained in a hyperplane. Hence, the fiber of $\Psi$ over $x$ consists of those $\Gamma$ that are contained in hyperplane sections of $C_{k-1} \to \pp^s$.
Since $C_{k-1} \to \pp^{s}$ is birationally very ample, we may use the Uniform Position Principle if the characteristic is $0$ or if $s \geq 4$ for any characteristic \cite{Borys}.
In these cases, the Uniform Position Principle says that the collection of $\Gamma$ contained in hyperplane sections of the image is irreducible. 

The cases where $s \leq 3$ occur when $m = 3$ or $4$. If $m = 3$, then $\vec{e} = (-2,0,0,1)$, which has $g' = 2$. If $m = 4$, then $\vec{e} = (-2, 0, 0, 2), (-2,-1,0,0,1),$ or $(-3, 0, 0,1)$, which have $g' = 2, 2,$ and $4$ respectively. In each of these cases $m \geq g'$.
\end{proof}

 \subsection{Necessary and sufficient conditions for very ampleness}
 The remainder of the paper is devoted to proving the following theorem. 

\begin{thm}
\label{Thm:VA}
A general line bundle in $\we$ is  very ample if and only if either 
  \begin{enumerate}
\item\label{va-1} $e_{k-2} \geq 0$ and $r = h^0(\O(\vec{e})) - 1 \geq 3$,
\item\label{va-2} $k = 3$ and $e_2 \geq 1$ and $e_3 - e_2 \leq 1$ and $\rho'(g,\vec{e}) = 0$,
\item\label{va-3} $k = 2$ and $e_{1} \geq 1$, 
\item\label{va-4} $k = 2$ and $(e_1,e_2) = (0,g)$ or $(e_1,e_2) = (0, g+1)$ (these are degree $2g +1$ and $2g+2$),
\item\label{va-5} $g = 0$ and $e_{k-1} \geq 0$,
\item\label{va-6} $g = 1$ and $\vec{e} = (-1, 0, 1)$ or $\vec{e} = (-1,\ldots,-1,0,0,0)$ (smooth plane cubic), or
\item\label{va-7} $g = 3$ and $\vec{e} = (-2, 0, 1)$ or $\vec{e} = (-2,-1,\ldots,-1,0,0,0)$ (smooth plane quartic).
\end{enumerate}
\end{thm}

The main content of the above theorem is the sufficiency of part \eqref{va-1}, which is
Theorem \ref{Thm:VA-intro} from the introduction. The proof of this more precise theorem relies on our precise characterization of relative very ampleness given in Theorem~\ref{rel-va}.
 Since a very ample line bundle is relatively very ample, we consider each of the cases in Theorem~\ref{rel-va} in the next five subsections.
 
 \subsubsection{Case \eqref{va-1} of Theorem \ref{rel-va}: when \(e_{k-2}\geq 0\)}  If $e_{k-2} \geq 0$ and \(r \geq 3\), then a generic line bundle in \(\W^{\vec{e}}_{\text{BN}}\) is very ample by Theorem \ref{Thm:VA-intro}.  So it remains to consider the case that \(e_{k-2} = e_{k-1} = e_k = 0\) and \(e_{k-3} < 0\).   By assumption, \(\deg \vec{e} \leq -k + 3\).  
In this case, \(|\L|\) is birationally very ample by Lemma~\ref{lem:3parts}\eqref{pt1} and maps \(C\) to \(\pp^2\).  This map is an embedding if and only if the genus of \(C\) is equal to the arithmetic genus of a plane curve of degree \(g - 1 + k + \deg(\vec{e})\).  That is, we must have the equality
\begin{equation}\label{eq:genus_plane}
 g =  \frac{(g + k + \deg \vec{e} - 2)(g + k + \deg \vec{e} - 3)}{2}.
 \end{equation}
 On the other hand, we know that \(\rho'(g, \vec{e}) \geq 0\) and hence 
 \begin{equation}\label{eq:rhop}
 g \geq u(\vec{e}) \geq \sum_{j < k-2 \leq i} e_i-e_j-1 = 3(-k+3 - \deg \vec{e}).
 \end{equation}
Combining \eqref{eq:genus_plane} and \eqref{eq:rhop} we obtain
\[
g  \geq \frac{(g -g/3 + 1)(g  -g/3)}{2} = \frac{(2g + 3)(2g)}{18}, 
\]
and hence \(g \leq 3\).  Those genera that satisfy \eqref{eq:genus_plane} are \(g=0\), which falls into Theorem \ref{Thm:VA}\eqref{va-5}, and \(g=1\), which falls into Theorem \ref{Thm:VA}\eqref{va-6}, and \(g=3\), which falls into Theorem \ref{Thm:VA}\eqref{va-7}.
 
Having dealt with this case, we will assume for the remainder of the proof that \(e_{k-2} < 0\).

\subsubsection{Case \eqref{va-2} of Theorem \ref{rel-va}: when $p = 0$} This case is not relevant, since we assume $p = 1$.

\subsubsection{Case \eqref{va-3} of Theorem \ref{rel-va}: when \(k=3\) and \(e_2 \geq 0\) and \(e_3 - e_2 \leq 1\) and \(\rho'(g, \vec{e}) = 0\)}

If \(e_2 \geq 1\), then the map from the nonnegative scroll \(\pp(E^\vee/F_-)\) to \(\pp^r\) is an embedding, and so the composition with the map from \(C\) to the nonnegative scroll (which we already know is an embedding by Theorem~\ref{rel-va}) is also an embedding.  This is Theorem~\ref{Thm:VA}\eqref{va-2}.

It remains to consider the case that \(e_2 = 0\).  First assume that \(e_2 = e_3 = 0\).  Then \(|\L|\) maps \(C\) to \(\pp^1\), and hence is very ample if and only if \(g=0\).  This therefore falls into Theorem~\ref{Thm:VA}\eqref{va-5}.

The only remaining case is \(e_2=0\) and \(e_3 = 1\).  In this case \(C\) must again be a smooth plane curve.  A similar calculation as above
 shows that \(g=1\) or \(g=3\), which fall into cases~\eqref{va-6} and \eqref{va-7} of Theorem~\ref{Thm:VA}.

\subsubsection{Case \eqref{va-4} of Theorem \ref{rel-va}: when \(k=2\) and \(e_1 \geq 0\)}

If \(e_1 \geq 1\), then as in the previous case, the map from the nonnegative scroll \(\pp(E^\vee/F_-) = \pp E^\vee\) to \(\pp^r\) is an embedding, and so the map from \(C\) to \(\pp^r\) is also an embedding.  This is Theorem~\ref{Thm:VA}\eqref{va-3}.

If \(e_1 = 0\), then since the directrix is contracted by the complete linear system of \(\O_{\pp E^\vee}(1)\), it suffices to determine when the curve meets the directrix at most once.  This is an intersection theory calculation on the Hirzebruch surface \(\pp(\O \oplus \O(-e_2))\).  If \(F\) denotes the class of a fiber and \(D\) denotes the class of the directrix, then \([C] = (e_2 + g + 1)F + 2D\). Hence, \(C \cdot D = 0\) if and only if \(e_2 = g+1\), and \(C \cdot D = 1\) if and only if \(e_2 = g\).  This is Theorem~\ref{Thm:VA}\eqref{va-4}.

\subsubsection{Case \eqref{va-5} of Theorem \ref{rel-va}: when \(g=0\) and \(e_{k-1} \geq 0\)}

In all of these cases the line bundle on \(\pp^1\) is of degree at least \(1\), and is hence very ample.  This is Theorem~\ref{Thm:VA}\eqref{va-5}.

\newpage

\bibliographystyle{amsplain.bst}
\bibliography{vampleness.bib}

\providecommand{\bysame}{\leavevmode\hbox to3em{\hrulefill}\thinspace}
\providecommand{\MR}{\relax\ifhmode\unskip\space\fi MR }
\providecommand{\MRhref}[2]{%
  \href{http://www.ams.org/mathscinet-getitem?mr=#1}{#2}
}
\providecommand{\href}[2]{#2}
\begin{thebibliography}{10}

\bibitem{CPJ1}
Kaelin Cook-Powell and David Jensen, \emph{Components of {B}rill-{N}oether loci
  for curves with fixed gonality}, Michigan Math. J. \textbf{71} (2022), no.~1,
  19--45. \MR{4389812}

\bibitem{CPJ2}
\bysame, \emph{Tropical methods in {H}urwitz-{B}rill-{N}oether theory}, Adv.
  Math. \textbf{398} (2022), Paper No. 108199. \MR{4372667}

\bibitem{EH83}
D.~Eisenbud and J.~Harris, \emph{Divisors on general curves and cuspidal
  rational curves}, Invent. Math. \textbf{74} (1983), no.~3, 371--418.
  \MR{724011}

\bibitem{im}
\bysame, \emph{Irreducibility and monodromy of some families of linear series},
  Ann. Sci. \'{E}cole Norm. Sup. (4) \textbf{20} (1987), no.~1, 65--87.
  \MR{892142}

\bibitem{farkas}
Gavril Farkas, \emph{Higher ramification and varieties of secant divisors on
  the generic curve}, J. Lond. Math. Soc. (2) \textbf{78} (2008), no.~2,
  418--440. \MR{2439633}

\bibitem{fl}
W.~Fulton and R.~Lazarsfeld, \emph{On the connectedness of degeneracy loci and
  special divisors}, Acta Math. \textbf{146} (1981), no.~3-4, 271--283.
  \MR{611386}

\bibitem{gp}
D.~Gieseker, \emph{Stable curves and special divisors: {P}etri's conjecture},
  Invent. Math. \textbf{66} (1982), no.~2, 251--275. \MR{656623 (83i:14024)}

\bibitem{bn}
P.~Griffiths and J.~Harris, \emph{On the variety of special linear systems on a
  general algebraic curve}, Duke Math. J. \textbf{47} (1980), no.~1, 233--272.
  \MR{563378 (81e:14033)}

\bibitem{JR}
David Jensen and Dhruv Ranganathan, \emph{Brill-{N}oether theory for curves of
  a fixed gonality}, Forum Math. Pi \textbf{9} (2021), Paper No. e1, 33.
  \MR{4199236}

\bibitem{Borys}
Borys Kadets, \emph{Sectional monodromy groups of projective curves}, J. Lond.
  Math. Soc. (2) \textbf{103} (2021), no.~1, 314--335. \MR{4203051}

\bibitem{kempf}
G.~Kempf, \emph{Schubert methods with an application to algebraic curves}, Pub.
  Math. Centrum, Amsterdam (1971).

\bibitem{kl}
S.~Kleiman and D.~Laksov, \emph{On the existence of special divisors}, Amer. J.
  Math. \textbf{94} (1972), 431--436. \MR{0323792}

\bibitem{LLV}
Eric Larson, Hannah Larson, and Isabel Vogt, \emph{Global {B}rill-{N}oether
  theory over the {H}urwitz space}, Geom. Topol., to appear.

\bibitem{interpolation}
Eric Larson and Isabel Vogt, \emph{Interpolation for {B}rill-{N}oether curves},
  \href{https://arxiv.org/abs/2201.09445}{arXiv:2201.09445}, 2022.

\bibitem{HannahRefined}
Hannah~K. Larson, \emph{A refined {B}rill-{N}oether theory over {H}urwitz
  spaces}, Invent. Math. \textbf{224} (2021), no.~3, 767--790. \MR{4258055}

\bibitem{P1}
\bysame, \emph{Universal degeneracy classes for vector bundles on {$\Bbb{P}^1$}
  bundles}, Adv. Math. \textbf{380} (2021), Paper No. 107563, 20. \MR{4200467}

\bibitem{sernesi}
Edoardo Sernesi, \emph{Deformations of algebraic schemes}, Grundlehren der
  mathematischen Wissenschaften [Fundamental Principles of Mathematical
  Sciences], vol. 334, Springer-Verlag, Berlin, 2006. \MR{2247603}

\end{thebibliography}

\end{document}